\documentclass[11pt,leqno]{amsart}
\usepackage{amssymb}
\usepackage{xypic}
\begin{document}
\theoremstyle{plain}
\newtheorem{thm}{Theorem}[section]
\newtheorem*{thm1}{Theorem 1}
\newtheorem*{thm2}{Theorem 2}
\newtheorem{lemma}[thm]{Lemma}
\newtheorem{lem}[thm]{Lemma}
\newtheorem{cor}[thm]{Corollary}
\newtheorem{propose}[thm]{Proposition}
\newtheorem{ex}[thm]{Example}
\theoremstyle{definition}
\newtheorem{rmk}[thm]{Remark}
\newtheorem{defn}[thm]{Definition}
\newtheorem{notations}[thm]{Notations}
\newtheorem{claim}[thm]{Claim}
\newtheorem{ass}[thm]{Assumption}
\numberwithin{equation}{section}
\newcounter{elno}      
\def\points{\list
{\hss\llap{\upshape{(\roman{elno})}}}{\usecounter{elno}}}
\let\endpoints=\endlist

%
%
%
\newcommand{\mc}{\mathcal} 
\newcommand{\mb}{\mathbb} 
\newcommand{\surj}{\twoheadrightarrow} 
\newcommand{\inj}{\hookrightarrow} \newcommand{\zar}{{\rm zar}} 
\newcommand{\an}{{\rm an}} \newcommand{\red}{{\rm red}} 
\newcommand{\Rank}{{\rm rk}} \newcommand{\codim}{{\rm codim}} 
\newcommand{\rank}{{\rm rank}} \newcommand{\Ker}{{\rm Ker \ }} 
\newcommand{\Pic}{{\rm Pic}} \newcommand{\Div}{{\rm Div}} 
\newcommand{\Hom}{{\rm Hom}} \newcommand{\im}{{\rm im}} 
\newcommand{\Spec}{{\rm Spec \,}} \newcommand{\Sing}{{\rm Sing}} 
\newcommand{\sing}{{\rm sing}} \newcommand{\reg}{{\rm reg}} 
\newcommand{\Char}{{\rm char}} \newcommand{\Tr}{{\rm Tr}} 
\newcommand{\Gal}{{\rm Gal}} \newcommand{\Min}{{\rm Min \ }} 
\newcommand{\Max}{{\rm Max \ }} \newcommand{\Alb}{{\rm Alb}\,} 
\newcommand{\GL}{{\rm GL}\,} 
\newcommand{\ie}{{\it i.e.\/},\ } \newcommand{\niso}{\not\cong} 
\newcommand{\nin}{\not\in} 
\newcommand{\soplus}[1]{\stackrel{#1}{\oplus}} 
\newcommand{\by}[1]{\stackrel{#1}{\rightarrow}} 
\newcommand{\longby}[1]{\stackrel{#1}{\longrightarrow}} 
\newcommand{\vlongby}[1]{\stackrel{#1}{\mbox{\large{$\longrightarrow$}}}} 
\newcommand{\ldownarrow}{\mbox{\Large{\Large{$\downarrow$}}}} 
\newcommand{\lsearrow}{\mbox{\Large{$\searrow$}}} 
\renewcommand{\d}{\stackrel{\mbox{\scriptsize{$\bullet$}}}{}} 
\newcommand{\dlog}{{\rm dlog}\,} 
\newcommand{\longto}{\longrightarrow} 
\newcommand{\vlongto}{\mbox{{\Large{$\longto$}}}} 
\newcommand{\limdir}[1]{{\displaystyle{\mathop{\rm lim}_{\buildrel\longrightarrow\over{#1}}}}\,} 
\newcommand{\liminv}[1]{{\displaystyle{\mathop{\rm lim}_{\buildrel\longleftarrow\over{#1}}}}\,} 
\newcommand{\norm}[1]{\mbox{$\parallel{#1}\parallel$}} 
\newcommand{\boxtensor}{{\Box\kern-9.03pt\raise1.42pt\hbox{$\times$}}} 
\newcommand{\into}{\hookrightarrow} \newcommand{\image}{{\rm image}\,} 
\newcommand{\Lie}{{\rm Lie}\,} 
\newcommand{\CM}{\rm CM}
\newcommand{\sext}{\mbox{${\mathcal E}xt\,$}} 
\newcommand{\shom}{\mbox{${\mathcal H}om\,$}} 
\newcommand{\coker}{{\rm coker}\,} 
\newcommand{\sm}{{\rm sm}} 
\newcommand{\tensor}{\otimes} 
\renewcommand{\iff}{\mbox{ $\Longleftrightarrow$ }} 
\newcommand{\supp}{{\rm supp}\,} 
\newcommand{\ext}[1]{\stackrel{#1}{\wedge}} 
\newcommand{\onto}{\mbox{$\,\>>>\hspace{-.5cm}\to\hspace{.15cm}$}} 
\newcommand{\propsubset} {\mbox{$\textstyle{ 
\subseteq_{\kern-5pt\raise-1pt\hbox{\mbox{\tiny{$/$}}}}}$}} 
\newcommand{\sA}{{\mathcal A}}
\newcommand{\sa}{{\mathcal a}} 
\newcommand{\sB}{{\mathcal B}} \newcommand{\sC}{{\mathcal C}} 
\newcommand{\sD}{{\mathcal D}} \newcommand{\sE}{{\mathcal E}} 
\newcommand{\sF}{{\mathcal F}} \newcommand{\sG}{{\mathcal G}} 
\newcommand{\sH}{{\mathcal H}} \newcommand{\sI}{{\mathcal I}} 
\newcommand{\sJ}{{\mathcal J}} \newcommand{\sK}{{\mathcal K}} 
\newcommand{\sL}{{\mathcal L}} \newcommand{\sM}{{\mathcal M}} 
\newcommand{\sN}{{\mathcal N}} \newcommand{\sO}{{\mathcal O}} 
\newcommand{\sP}{{\mathcal P}} \newcommand{\sQ}{{\mathcal Q}} 
\newcommand{\sR}{{\mathcal R}} \newcommand{\sS}{{\mathcal S}} 
\newcommand{\sT}{{\mathcal T}} \newcommand{\sU}{{\mathcal U}} 
\newcommand{\sV}{{\mathcal V}} \newcommand{\sW}{{\mathcal W}} 
\newcommand{\sX}{{\mathcal X}} \newcommand{\sY}{{\mathcal Y}} 
\newcommand{\sZ}{{\mathcal Z}} \newcommand{\ccL}{\sL} 
\newcommand{\A}{{\mathbb A}} \newcommand{\B}{{\mathbb 
B}} \newcommand{\C}{{\mathbb C}} \newcommand{\D}{{\mathbb D}} 
\newcommand{\E}{{\mathbb E}} \newcommand{\F}{{\mathbb F}} 
\newcommand{\G}{{\mathbb G}} \newcommand{\HH}{{\mathbb H}} 
\newcommand{\I}{{\mathbb I}} \newcommand{\J}{{\mathbb J}} 
\newcommand{\M}{{\mathbb M}} \newcommand{\N}{{\mathbb N}} 
\renewcommand{\P}{{\mathbb P}} \newcommand{\Q}{{\mathbb Q}} 

\newcommand{\R}{{\mathbb R}} \newcommand{\T}{{\mathbb T}} 
\newcommand{\U}{{\mathbb U}} \newcommand{\V}{{\mathbb V}} 
\newcommand{\W}{{\mathbb W}} \newcommand{\X}{{\mathbb X}} 
\newcommand{\Y}{{\mathbb Y}} \newcommand{\Z}{{\mathbb Z}} 

\title{Hilbert-Kunz density functions and $F$-thresholds}
\author{Vijaylaxmi Trivedi and Kei-Ichi Watanabe}
\date{}
\address{School of Mathematics, Tata Institute of Fundamental Research, Homi Bhabha Road, Mumbai-40005, India }
\email{vija@math.tifr.res.in}
\address{Department of Mathematics, College of Humanities and Sciences, Nihon University, Setagaya-Ku, Tokyo 156-0045, Japan}
\email{watanabe@math.chs.nihon-u.ac.jp}
\thanks{}
\subjclass{}
\begin{abstract}We had shown earlier that  
for a standard graded ring $R$ and a graded ideal $I$ in characteristic $p>0$, 
 with $\ell(R/I) <\infty$,  
 there exists a  
compactly supported continuous function $f_{R, I}$ whose Riemann  integral is 
the HK multiplicity $e_{HK}(R, I)$.
We explore further some other invariants, namely the shape of 
the graph of  $f_{R, {\bf m}}$ (where ${\bf m}$ is the graded maximal ideal of $R$) 
 and the maximum support
(denoted as $\alpha(R,I)$) of $f_{R, I}$.

In case $R$ is a  domain of dimension $d\geq 2$, we prove that $(R, {\bf m})$ is a regular ring 
if and only if $f_{R, {\bf m}}$ has a  symmetry 
$f_{R, {\bf m}}(x) = f_{R, {\bf m}}(d-x)$,
for all $x$.

If $R$ is strongly $F$-regular on the punctured spectrum then  
we prove that the $F$-threshold $c^I({\bf m})$ coincides with $\alpha(R,I)$.

As a consequence, if $R$ is a two dimensional domain and $I$ is generated by 
homogeneous elements of the same degree, then 
we have (1) a formula for the $F$-threshold $c^I({\bf m})$ 
 in terms
of the minimum  strong Harder-Narasimahan slope of the syzygy bundle and  (2)
 a well defined notion of the  $F$-threshold $c^I({\bf m})$
 in characteristic $0$.

This characterisation readily computes  $c^{I(n)}({\bf m})$, 
for the set of all irreducible 
plane trinomials $k[x,y,z]/(h)$, where ${\bf m} = (x,y,z)$ and 
$I(n) = (x^n, y^n, z^n)$.

\end{abstract}

\maketitle
\section{Introduction}
Let $(R, I)$ be a standard graded pair, {\em i.e.}, $R$ is a Noetherian 
standard graded ring  over a perfect field $k$ (unless otherwise stated) of 
characteristic $p >0$ 
and $I$ is a graded ideal of finite colength. 
 Let ${\bf m}$ be the graded maximal ideal of $R$.

If $M$ is a  finitely generated graded $R$-module 
then we have 
a compactly supported continuous function $f_{M, I}:[0, \infty)\longto [0, \infty)$
called  the Hilbert-Kunz  density function (see [T2]) for $(M,I)$ (we henceforth 
abbreviate the term `Hilbert-Kunz' to the term `HK').
  We 
realize this function as the limit of a uniformly convergent sequence of 
compactly supported functions 
$\{f_n(M, I):\R\to [0, \infty)\}_{n\in \N}$, where  
$$f_{n}(M, I)(x) = \frac{1}{q^{d-1}}\ell((M/I^{[q]}M)_{\lfloor xq\rfloor}),
~~\mbox{for}~~q=p^n.$$

Moreover $$\int_0^{\infty}f_{M, I}(x)dx = 
e_{HK}(M, I),$$
 where $e_{HK}(M, I)$ denotes 
the famous invariant (introduced by P. Monsky [M1]) called the
 HK multiplicity of $M$ with 
respect to $I$.

Since  the function $f_{M,I}$ is the 
uniformly  convergent limit of  the sequence $\{f_n(M,I)\}_n$,
and is also  an additive and a multiplicative function (the  
HK density functions of the rings explicitly gives the HK density function of 
 their Segre product),  it is a versatile tool 
  to handle  $e_{HK}(M, I)$.

In this paper we study the shape  of the graph of $f_{R,I}$, and the
 maximal support $\alpha(R,I)$  of $f_{R,I}$, where 
 $$\alpha(R, I) =  
\mbox{Sup}~\{x\mid f_{R, I}(x)\neq 0\}.$$

First we prove (in Theorem~\ref{c1} and Theorem~\ref{shk}) that the shape of 
the graph of $f_{R, {\bf m}}$   (and also the invariant  $\alpha(R, I)$) 
determines the 
regularity of $(R, {\bf m})$:
\vspace{5pt}

\noindent{\bf Theorem~A}.~~{\it If $(R, {\bf m})$ is a standard graded 
domain  of dimension $d\geq 2$ then 
\begin{enumerate}
\item $f_{R, {\bf m}}(x) = f_{R, {\bf m}}(d-x)$, for all 
$x$ if and only if the ring $(R, {\bf m})$ is regular. 
\item
In fact $\alpha(R, {\bf m})   
= d$ if and only if 
$(R, {\bf m})$ is a regular ring.
\item If $\dim~R = 2$ then either   
\begin{enumerate} 
\item  $f_{R, {\bf m}}$ is 
symmetric, {\em i.e.}, 
$f_{R, {\bf m}}(1-y) = f_{R, {\bf m}}(1+y)$, for all $y\in (0, 1)$, or
\item  $f_{R, {\bf m}}(1-y) > f_{R, {\bf m}}(1+y)$, for all $y\in (0, 1)$. 
\end{enumerate}
\end{enumerate}}

Next we relate the invariant 
$\alpha(R, I)$  to  $c^I({\bf m})$,
the $F$-threshold of ${\bf m}$ with respect to the ideal $I$.

Recall that 
the $F$-thresholds were introduced and studied in [MTW], in the 
case of regular rings. In a more general setting (when $R$ is not 
regular) it was further studied in [HMTW].
In [MTW] (Question 1.4) the following question  was posed.  

\vspace{5pt}

{\bf Question}.\quad Is it true that for all nonzero ideals $J$ and $I$ with 
$J \subseteq  \mbox{Rad}(I) \subseteq  {\bf m}$, the $F$-threshold $c^I(J)$ is a 
rational number?
\vspace{5pt}
 
For regular rings, one gets a  positive answer from a 
 series of papers  ([KLZ],  [BMS1], [BMS2]) in the following way: the 
$F$-thresholds of $I$  
are also  $F$-jumping numbers of $I$, and  all  $F$-jumping numbers are rational.

 Moreover if  $R$ is a direct summand of a regular $F$-finite domain $S$,
then, by Proposition~4.17 of [AHN], 
$c^{I}(J)$ is a rational number. Here 
$c^I(J)$ was  identified with $c^{IS}(JS)$ and hence is an  $F$-jumping number of 
$JS$.
We recall that in case $(R, {\bf m})$ is regular, either local or standard graded, 
and  $J\subset R$ an ideal, then   
$c^{\bf m}(J)$ is the  first jumping number $\mbox{fpt}(J)$  
(this is called the $F$-pure threshold of $J$).

However, in singular cases, $F$-thresholds may differ from $F$-jumping numbers, 
for example (1) when $R$ is  the coordinate ring of the 
Segre product $\P^m\times \P^n$, where $m\neq n$,  we have 
$\mbox{fpt}({\bf m}) < c^{\bf m}({\bf m})$ (see [CM] and [HWY]), (2) when $R =
k[x,y,z]/(xy-z^2)$, 
we have  $\mbox{fpt}({\bf m}) =1 < c^{\bf m}({\bf m}) =3/2$ (see [TW] and [HMTW]).

In general, to the best of our knowledge, 
it is not known  whether  $c^{\bf m}({\bf m})$  is rational, 
even in  graded cases.

As  a  consequence of identifying $c^I({\bf m})$ with $\alpha(R, I)$ 
we prove that   the $F$-thresholds 
$c^I({\bf m})$ are rational numbers in the cases listed 
in Theorem~B below.

\vspace{5pt}

\noindent{\bf Theorem~B}.~~{\it For a standard graded pair 
$(R, I)$, where $R$ is a two dimensional domain,
 the following statements hold.  
 \begin{enumerate}
\item If $I$ is generated by homogeneous elements of the same degree then 
$c^I({\bf m}) = 1 - a_{min}(V)/d$ and hence is a rational number. 
\item  If $R$ is normal then $c^I({\bf m})$ 
 is a rational number.
\end{enumerate}}

Here  $V$ denotes   a  syzygy bundle   associated to 
the pair $(R, I)$ (as in  Notations~\ref{n3}) and $a_{min}(V)$ denotes the minimum strong 
HN slope (as in Notations~\ref{hnd}) and $d = e_0(R,{\bf m})$ denotes 
the Hilbert-Samuel 
multiplicity of $R$.

In another work [T5], the first author has used 
this  explicit formulation of $c^I({\bf m})$ in terms of the strong HN 
slopes of its syzygy  bundle (along with a  construction of D. Gieseker [G]) to
 give 
an example of a  set of $F$-thresholds of an ideal with
 accumulation points, which answers another question by
Musta\c{t}\u{a}-Takagi-Watanabe (Question~2.11 in [MTW]).

Next we look at the {\em reduction mod} $p$ behaviour of 
these  $F$-thresholds. Recall that
Theorem~3.4 and Proposition~3.8 of [HY] describe the behaviour of 
$c^{\bf m}(I)$ under {\em reduction mod} $p$ (note that in this case 
 $c^{\bf m}(I) = \mbox{fpt}_{\bf m}(I)$) as follows:

\vspace{5pt}

\noindent{\bf Theorem~[HY]}.~~{\it
If $R = A[X_1, \ldots, X_d]$ is a polynomial ring over $A$, where $A$ is a localization of $\Z$ at some nonzero 
integer and $I\subset {\bf m} = (X_1, \ldots, X_d)$ is an ideal then 
\begin{enumerate}\item
$\lim_{p\to \infty} c^{{\bf m}_p}(I_p)$ exists and 
 $\lim_{p\to \infty} c^{{\bf m}_p}(I_p) = 
\mbox{lct}_{\bf m}(I)$, where $\mbox{lct}_{\bf m}(I)$ is the 
log canonical threshold of $I_{\Q}$ at ${\bf m}_{\Q}$.
\item Moreover for 
$p>>0$, we have $\mbox{lct}_{\bf m}(I) = 
\lim_{p\to \infty} c^{{\bf m}_p}(I_p) \geq  c^{{\bf m}_p}(I_p).$
\end{enumerate}}

In dimension two, the formulation of $c^I({\bf m})$ in terms of
the strong HN slopes of a syzygy bundle (as in Theorem~B)
gives  
\begin{enumerate}
\item a well defined notion of 
 $F$-threshold in characteristic $0$ and 
\item 
a characterization of  the  
strong semistability behaviour  of the syzygy bundle $V_s$ 
({\em reduction mod}~$p_s$ of $V$) in terms of the $F$-threshold  
 $c^{I_s}({\bf m}_s)$, of ${\bf m}_s$ with respect to $I_s$ 
(see Definition~\ref{ds}).
\end{enumerate}

This is done using the following result 
(proved in subsection~6.3):
\vspace{5pt}

\noindent{\bf Theorem~C}.~~{\it Let $(R, I)$ be a  standard graded pair 
where $R$ is a  two dimensional domain 
in characteristic~$0$ and where $I$ is generated by homogeneous elements of 
the same degree. If $(R_s, I_s)$ is the  {\em reduction mod} $p_s$ of the pair 
$(R,I)$, (obtained from a spread $(A, R_A, I_A)$, where $p_s= \Char~R_s$), then 
\begin{enumerate}
\item
 $$c_{\infty}^I({\bf m}):= \lim_{p_s\to \infty}
c^{I_s}({\bf m}_s)\quad\mbox{exists and}$$
\item $c^I_{\infty}({\bf m}) = 
\mbox{Sup}\{x\mid f^{\infty}_{R, I}(x)\neq 0\}$, where 
$f^{\infty}_{R, I}(x) = \lim_{p_s\to \infty}f_{R_s, I_s}(x)$.
\item  $c^{I_s}({\bf m}_s) \geq c_{\infty}^I({\bf m})$ if $p_s>>0$.
If, in addition, the bundle $V$ is semistable 
and $V_s$ is the 
{\em reduction mod}~$p_s$ of $V$  then  
$$c^{I_s}({\bf m}_s) = c_{\infty}^I({\bf m})~~
\iff~~ V_{s}~~\mbox{is
strongly semistable}.$$ 
\end{enumerate}}

In particular we have

\vspace{5pt}

\noindent{\bf Corollary~D}.~~{\em Let $(R, {\bf m})$ be a standard graded pair, 
where $R$ is a two dimensional domain in characteristic $0$ and let 
$X=\mbox{Proj}~R$ with  
 $\deg\sO_X(1) > 2\mbox{genus} (X)$,  then 
 for $p_s>>0$, 
$$c^{{\bf m}_s}({\bf m}_s) = 
c_{\infty}^{\bf m}({\bf m}) \iff V_{s}~~\mbox{is 
strongly semistable}.$$}

We  note that, in contrast to the Theorem~[HY], here we have the 
following reverse inequality:
$$\mbox{for}~~ p_s >>0~~\mbox{we have}~~ c_{\infty}^I({\bf m}) =
 \lim_{p_s\to \infty} c^{I_s}({\bf m}_s) \leq c^{I_s}({\bf m}_s).$$

For higher dimensional cases, we relate 
 the two invariants $c^I({\bf m})$ and 
$\alpha(R, I)$  (in the subsection~4.1), which leads us to  
ask the following natural question

\vspace{5pt}
\noindent{\bf Question}.~~
Let $(R, I)$ be a standard graded pair and ${\bf m}$ be the graded maximal ideal 
of $R$. Then, is
 $\alpha(R, I) = c^I({\bf m})$?
\vspace{5pt}

The following theorem (proved in the subsection~4.1) 
lists the cases where we show that the answer is  affirmative

\vspace{5pt}

\noindent{\bf Theorem~E}.~~ Let $(R, I)$ be a standard graded pair
 and ${\bf m}$ be the graded maximal ideal 
of $R$. Then 
{\it \begin{enumerate}
\item  $\alpha(R, I) \leq c^I({\bf m})$, for dimension~$R = d\geq 2$. Moreover
\item  the equality $\alpha(R, I) = c^I({\bf m})$ holds if the pair 
$(R, I)$ satisfies one of the following conditions:
\begin{enumerate}
\item $I$ is generated by a  system of parameters,  
\item $R$ is  strongly $F$-regular on the punctured spectrum 
$\Spec R\setminus \{{\bf m}\}$, or
\item $R$ is a two dimensional domain and $I$ is generated by homogeneous 
elements of the same degree.
\end{enumerate}
Moreover, the  equality holds for the Segre products of all such pairs.
\end{enumerate}}

\vspace{5pt}
At the end of the paper, in Section~7, we give an 
explicit formula for $c^{I(n)}({\bf m})$, where 
$R = k[x,y,z]/(h)$
are irreducible plane trinomials  and  where $I(n) = (x^n, y^n, z^n)$. 

If we denote $c^{I(n)}({\bf m})$ by $c^{{I(n)}_p}({\bf m}_p)$ where  
$p = \mbox{char}~k$, then in fact we find that 
$c^{{I(n)}_p}({\bf m}_p)$ (as $p$ varies) 
is just
a function of the congruence class of $p~\mbox{{\em mod}}~2\lambda_h$,
where $\lambda_h$ is an explicit computable  integer given in terms of the exponents of the 
trinomial $h$ (see Notations~\ref{nt}).
In particular
 we have the following 
\vspace{5pt}

{\bf Example}.~~{\it
If $R=k[x,y,z]/(h)$ is an irreducible trinomial of degree $d\geq 3$
 then 
$$\mbox{for all}~~p\geq d^2~~\mbox{and}~~p\equiv\pm 1\pmod{2\lambda_h}
~~\mbox{we have}~~c^{{I(n)}_p}({\bf m}_p) = c_{\infty}^{I(n)}({\bf m}).$$
If  $(R,{\bf m})$ is  a Segre product of any finitely many 
irreducible trinomials then 
\begin{enumerate}
\item  there are 
infinitely many primes $p>0$, for which 
$c^{{I(n)}_p}({\bf m}_p) = c_{\infty}^{I(n)}({\bf m})$.
\item  Moreover, if one of the trinomials, occuring  in the product,   
is a symmetric curve ({\em i.e.}, $h = x^{d-a}y^a+y^{d-a}z^a+z^{d-a}x^a$) 
of degree $d > 5$ then 
there are also
infinitely many primes $p>0$, for which 
$c^{{\bf m}_p}({\bf m}_p) > c_{\infty}^{\bf m}({\bf m})$.
\end{enumerate}}
We recall the known computations made for some explicit polynomials to 
demonstrate
the complexity of $c^{I}(J)$.

 When $R= k[x,y]$ and $f= x^2+y^3$, or when  $R= k[x,y,z]$ and $f$ is a 
homogeneous polynomial of degree~$3$ with 
isolated singularity at $(x,y,z)$ then $c^{\bf m}(f)$ was computed and such 
phenomena were exhibited in 
 Examples~4.3 and 4.6 of [MTW].
In Corollary~3.9, Hara and Monsky (see [H])  independently described (using sygygy gaps)
the possible values of $c^{(x,y)}(f)$,
whenever $f\in k[x,y]$
is homogeneous of degree 5 with an isolated singularity at
the origin, when $p\neq 5$. Theorem~4.2 of [Vr] computes $c^{\bf m}({\bf m})$, 
for  diagonal hypersurfaces.

\vspace{5pt}

The organization of this paper is as follows.

In Section~3, we compute the HK density function for $(R, I)$, where 
$I$ is generated by a system of parameters. 
This turns out to be  a volume  function, depending only on the 
degrees of the generators of $I$.
 Here we use the uniform convergence property of the sequence $\{f_n(R, I)\}_n$ 
and the fact that $f_{R, I}$ is a continuous function.

In Section~4, we relate $\alpha(R, I)$ with the $F$-threshold 
$c^I({\bf m})$ and give the examples of the cases when the equality
$\alpha(R, I) = c^I({\bf m})$ does hold.
Here we also
characterize the regularity property of the ring $R$  in terms of the shape of the 
graph of $f_{R, {\bf m}}$ and also in terms of the number  $\alpha(R, {\bf m})$.

From Sections~5 onwards we restrict to  standard graded pairs in dimension two.
In Section~5, we list some (known)  results  
about vector bundles over nonsingular 
projective curves, which we use later in this paper.

In Section~6, we give a  notion (analogous to $f_{R, I}$) of the HK density function
$f_{V, \sO_X(1)}$, for a 
pair $(V, \sO_X(1))$, where $V$ is a vector-bundle on a nonsingular 
curve $X$ and $\sO_X(1)$ is a very ample line-bundle on $X$.
Then we relate $f_{R,I}$ with the HK density functions of the syzygy 
vector bundles which are expressed in terms ot their strong HN datum.

In Section~7, we give  the computations of $F$-thresholds for 
 plane trinomials. 

The results stated in Theorem~\ref{vb1} can be generalized by removing 
the hypothesis that the ideal $I$ be generated by homogeneous  elements of 
same degree elements. However
the  arguments are technical and will appear in a subsequent paper.

The authors thank the referee for careful reading of the manuscript and 
providing various suggestions which greatly improved the exposition of the paper.

\section{preliminaries}
Let $(R, I)$ be a standard graded pair over a perfect field of characteristic
$p>0$. 
Let $M$ be a finitely generated 
graded $R$-module.
We recall the following known properties of $f_{M,I}$ from
[T2].

\begin{enumerate}
\item {\underline {Additive property}}: Like  HK multiplicity, the HK density function 
too have the additive property, which  reduces the theory of  
$f_{M, I}$ to the theory of
$f_{R,I}$, where $R$ is a normal domain:
Let $\Lambda$ be the set of minimal 
prime ideals $P$ of $R$ such that $\dim R/P = \dim R$.
Then 
$$f_{M,I} = \sum_{P\in \Lambda}f_{R/P, I}\lambda(M_P).$$
As a consequence, we have
\begin{enumerate}
\item $f_{M, I} = 0$, if $\dim~M < \dim~R$.
\item If $R$ is an integral domain then $f_{R, I} = f_{S, IS}$, where 
$S$ is the  normalization of $R$, regarded as a graded $R$-module. 
\item $f_{M, I} = f_{M(n), I}$, for every $n\in \Z$.
\end{enumerate}

\item {\underline {Multiplicative property}}: The multiplicative property expresses 
the HK density function of the Segre product of rings in terms 
of the HK density function of the individual rings:
If  $(R, I)$ and $(S, J)$ are two pairs and
$F_{R}(x) = e(R)x^{d-1}{(d-1)!}$, where
$e(R)$ denotes the Hilbert-Samuel multiplicity of $R$ with 
respect to its irrelevant maximal ideal ${\bf m}$ and $d = \dim~R$
 then the Segre product $(R\#S, I\#J)$ 
satisfies
$$ F_{R\#S}(x)-f_{R\#S, I\# J}(x)  =  
\left[F_{R}(x) - f_{R, I}(x)\right]
\left[F_{S}(x) - f_{S,J}(x)\right].$$

\item Let $I \subseteq I'$ such that $I'$ is homogeneous 
then 
$$e_{HK}(R, I) = e_{HK}(R, I') \iff f_{R, I}(x) = 
f_{R, I'}(x),~~\mbox{for all}~~x.$$
In particular, if $R$ is equidimensional then   
$$f_{R, I} = f_{R, I'} \iff I'\subseteq I^*,$$
where $I^*$ denotes the tight closure of $I$ in $R$.
\item If $n_0\in \N$ such that ${\bf m}^{n_0}\subseteq I$ and the ideal $I$ is 
generated by $\mu$ generators then 
$\mbox{the support of}~f_{R, I}\subseteq [0, n_0\mu]$.

\end{enumerate}

\section{HK density functions for parameter ideals}
\subsection{HK density functions for parameter ideals}
Here we give an explicit formula for the HK density function $f_{R, I}$, when
$I$ is generated by a system of parameters. As expected, we find that
$f_{R, I}$ solely depends on the degrees of the generators of $I$.

\begin{defn}\label{d4}Given nonnegative integers $n_1, \ldots, n_m$, consider
a $m$-parallelotope $P = [0, n_1]\times \cdots\times [0, n_m]$. We define 
a volume function 
$$V_{m-1}(n_1, \ldots, n_m):[0, \infty)
\longto [0, \infty)~~~\mbox{given by}~~~ 
x\to \mbox{Vol}_{m-1}(P\cap H_x),$$
where $H_x = \{(y_1, \ldots, y_m)\in \R^m \mid \sum_iy_i 
= x\}$ is a $m-1$-dimensional hyperplane in $\R^m$ and 
$\mbox{Vol}_{m-1}$ is the  $(m-1)$-dimensional Euclidean volume. 

\end{defn}
\begin{lemma}\label{l1} Let $(R,I)$ be a standard graded pair, where $I$ is 
generated by homogeneous system of parameters $f_1, \ldots, f_d$ of 
degree $n_1,n_2, \ldots, n_d$ respectively. Then 
$$f_{R, I}(x) = e(R)V_{d-1}(n_1, \ldots, n_d)(x),$$
where the function $V_{d-1}(n_1, \ldots, n_d)$ is given as in Definition~\ref{d4}
and $e(R)$ is the Hilbert-Samuel multiplicity of $R$ with respect to the ideal 
${\bf m}$.
\end{lemma}
\begin{proof}By the additive property of the HK density function and 
the Hilbert-Samuel multiplicity, we can assume that 
$R$ is a normal domain and hence $f_1$ is a non zero-divisor on $R$.
 For $d=2$ the lemma is  easy to check as $\{f_1, f_2\}$ form a 
regular sequence. Henceforth we assume $d\geq 3$.
We prove the lemma by induction on $d$.  
For the ring $S= R/f_1R$ and the ideal $J = I/f_1R$
$$f_{S, J}(x) = e(S, {\bf m}S)V_{d-2}(n_2, \ldots, n_d)(x),~~\mbox{for all}~~x\in 
\R,$$ 
where,   $e(S,{\bf m}S) = n_1 e(R)$. 
For every $k\geq 1$ and $q=p^n$, there exists  the canonical degree $0$ surjective  map of 
graded $R$-modules, where $J_1 = f_2R+\cdots +f_dR$.
\begin{equation}\label{e8}
\frac{S}{J^{[q]}}(-kn_1)  \longto 
\frac{f_1^kR}{f_1^{k+1}R+
f_1^kR\cap J_1^{[q]}} = \frac{f_1^kR+J_1^{[q]}}{f_1^{k+1}R+J_1^{[q]}}.
\end{equation}
Hence, for any $x\geq 0$, 
we have the surjective map
$$\oplus_{k=0}^{q-1}\left(\frac{S}{J^{[q]}}\right)_{-kn_1+\lfloor xq\rfloor}\longto 
\left(\frac{R}{I^{[q]}}\right)_{\lfloor xq\rfloor},$$
which gives
\begin{equation}\label{e7}
f_n(R,I)(x) \leq \frac{1}{q}\sum_{k=0}^{q-1}f_n(S, J) 
\left(\frac{\lfloor xq\rfloor-kn_1}{q}\right).\end{equation}
For the sake of brevity,  throughout the rest of the proof,  
we denote $f_n(S, J)$ by $g_n$.
 For every $q =p^n$, $k\geq 1$ and for every 
$kn_1/q < \lambda \leq (kn_1+n_1)/q$ (applying Lemma~2.8 of [T2] 
to an injective  graded map  $S_{\lfloor xq-\lambda q\rfloor} \longto 
S_{\lfloor xq\rfloor-kn_1}$) we have  
$$g_n\left(\frac{\lfloor xq\rfloor-kn_1}{{q}}\right) =
g_n\left(\frac{\lfloor xq\rfloor}{q} -\lambda\right) + O({1}/{q}).$$

Hence 
$$ \left(\frac{1}{q}\right)~g_n\left(\frac{\lfloor xq\rfloor -kn_1}{q}\right) = 
 \frac{1}{n_1}
\int_{kn_1/q}^{(k+1)n_1/q}g_n\left(\frac{\lfloor 
xq\rfloor}{q}-\lambda\right)d\lambda + O(1/q^2)
$$

\vspace{5pt}
\noindent{\underline{Case}~(1)}.\quad If $n_1 <x$. Then for $q>>0$, we have 
$(q-1)n_1\leq \lfloor xq\rfloor$.

Therefore
$$\mbox{R.H.S. of}~~(\ref{e7})  
= ({1}/{q})\sum_{k=0}^{q-1}g_n\left(\frac{\lfloor xq\rfloor -kn_1}{q}\right)
= ({1}/{n_1})\int_{{\lfloor xq\rfloor}/{q}-
n_1}^{{\lfloor xq\rfloor}/{q}}
g_n\left(\lambda\right)d\lambda
+O\left(1/q\right).$$

Now taking limit for (\ref{e7})  as $q\to \infty$, and by  
induction on $d$, we get
$$f_{R,I}(x) \leq ({1}/{n_1})\int_{x-n_1}^xf_{S, J}(\lambda)d\lambda
= \frac{1}{n_1} e(S, {\bf m}S)\int_{x-n_1}^x 
V_{d-2}(n_2, \ldots, n_d)(\lambda)d\lambda.$$
Since $n_1 <x$, this gives
$f_{R,I}(x) \leq e(R) V_{d-1}(n_1, n_2,\ldots, n_d)(x).$

\noindent{\underline{Case}~(2)}\quad If $n_1\geq x$ then 
$\lfloor xq\rfloor = n_1{\tilde m}+r$, where $0\leq r<n_1$ and 
${\tilde m} <q-1$.
$$\mbox{R.H.S. of}~~(\ref{e7})
= ({1}/{q})\left[
g_n(\frac{\lfloor xq\rfloor -{\tilde m}n_1}{q})+
g_n(\frac{\lfloor xq\rfloor -({\tilde m}-1)n_1}{q})+
\cdots g_n(\frac{\lfloor xq\rfloor}{q})\right]$$
$$= ({1}/{n_1})\int_0^{{\tilde m}n_1/q}g_n(\lambda)d\lambda +
{\tilde m}O({1}/{q^2}).$$
Now taking limit for (\ref{e7})  as $q\to \infty$,  we get
$$f_{R,I}(x) \leq ({1}/{n_1})\int_{0}^x f_{S, J}(\lambda)d\lambda 
= e(R)V_{d-1}(n_1, \ldots, n_d)(x).$$
Hence $f_{R, I}(x)\leq e(R)V_{d-1}(n_1, \ldots, n_d)(x)$ for all $x\in \R$. 

But
$e(R)V_{d-1}(n_1, \ldots, n_d)(x)-f_{R, I}(x)$ is a nonnegative continuous function with 
integral $=0$. 
Therefore  $f_{R, I}(x) = e(R)V_{d-1}(n_1, \ldots, n_d)(x)$, for all $x$. 
\end{proof}

\begin{cor}\label{p1}Let $(R, I)$ be a pair as above. If
$I$ is a parameter ideal of $R$ generated by elements of degrees, say $n_1, \ldots, 
n_d$
then $f_{R,I}$ is a symmetric function around $n_1+\cdots +n_d$, {\it i.e.},
$$f_{R, I}(x) = f_{R, I}(n_1+\cdots + n_d-x),~~~\mbox{for all}~~x\geq 0.$$
\end{cor} 

\section{The function $f_{R,I}$ versus the $F$-threshold
$c^I({\bf m})$  and 
the regularity} 

\subsection{Support of the HKd function}

In this subsection we compare the maximum support of the HK density function
$f_{R, I}$ and the $F$-threshold $c^I({\bf m})$ of ${\bf m}$ with respect to $I$.
We discuss  the cases, where we can show that both the invariants coincide.
  
\begin{defn}\label{d1}For a standard graded pair $(R,I)$ 
 and a finitely generated graded $R$-module $M$, 
let
$$\alpha(M, I) = \mbox{Sup}~\{x\mid f_{M,I}(x)>0\}.$$ 
\end{defn}

\begin{rmk}\label{r1}If $(R, I)$ is  a standard 
graded pair of dimension $d\geq 2$ and $I$ is 
generated by homogeneous elements of degrees  $d_1 \leq d_2 < \cdots $ then 
$\alpha(R, I) >d_1$:
By definition, the function $f_{R, I}(x) = 
e(R)x^{d-1}/((d-1)!)$, 
for $0\leq x \leq d_1$.
In particular, $f_{R, I}\mid_{[0,d_1]}$ is a strictly monotonic increasing function 
and  $\alpha(R, I)$ is a positive real number with $\alpha(R, I) > d_1$.
 \end{rmk}

We recall the following notion of $F$-threshold, as defined  in [HMTW].
 
\begin{defn}\label{d2}Let $I$ and $J$ be two ideals such that 
$J \subseteq \sqrt{I}$.
Then  the
 $F$-threshold of $J$ with respect to $I$ is 
$$c^{I}(J)  = \lim_{q\to \infty}
\frac{\mbox{min}~\{r\mid {J}^{r+1}\subseteq {I}^{[q]}\}}{q}\quad\mbox{if it exists}.$$
\end{defn}
The existence of the above limit in full generality was proved in 
  [DsNbP].

\begin{propose}\label{p2} Let $(R, I)$ be a standard graded pair of dimension 
$\geq 2$. 
 Then 
$\alpha(R, I)\leq c^I({\bf m})$. 
\end{propose}
\begin{proof}
Let $c^I({\bf m}) = c$. Then,
 given $\epsilon >0$, there is a 
$q(\epsilon)$ such that for all $q \geq q(\epsilon)$, we have 
${\bf m}^{(c+\epsilon)q} \subseteq I^{[q]}$. Since $R$ is a standard graded ring 
this implies $(R/I^{[q]})_m = 0$, for $m \geq (c+\epsilon)q$.

Therefore, for all $x \geq c+\epsilon$ and for $q\geq q(\epsilon)$, 
$$ \ell(R/{I}^{[q]})_{\lfloor xq\rfloor} = 0\implies 
f_n(x) = \frac{1}{q^{d-1}}
\ell(R/{I}^{[q]})_\lfloor xq\rfloor  = 0,$$ 
 Hence, for every $\epsilon >0$, $f(x) = 
\lim_{n\to \infty} f_n(x) = 0$, for all $x\geq (c+\epsilon)$.

Since $f:[0, \infty)\longto [0, \infty)$ is a continuous function,  we deduce that 
$f(x) = 0$ for all $x\geq c$.
\end{proof}

Next we discuss the cases where $\alpha(R,I) = c^I({\bf m})$.

\begin{lemma}\label{r31} For a standard graded pair $(R, I)$ of dimension $1$, 
where $R$ is a reduced ring, we have $\alpha(R,I) = c^I({\bf m})$.\end{lemma}

\begin{proof}
 We recall (Theorem~2.9 [T2]) that $f_{R,I}$  is the pointwise 
limit of $f_n(R,I)$ (here the convergence may not be a uniform convergence). 
Moreover, for given  $x\geq 0$, there is $n_0$ such that for all 
$q = p^n\geq p^{n_0}$,  we have 
$$f_{R, I}(x) = f_n(R, I)(x) = \ell(R/I^{[q]})_{\lfloor xq\rfloor}.$$ 
Hence, for $x\geq \alpha(R, I)$, 
$$f_{R, I}(x) = 0 \implies {\bf m}^{\lfloor xq\rfloor} \subseteq I^{[q]},~~ 
\mbox{for}~~ q\geq p^{n_0} \implies c^I({\bf m})\leq x \implies c^I({\bf m}) 
= \alpha(R, I).$$
\end{proof}

\begin{propose}\label{p3} If $(R, I)$ is a standard graded pair where 
$R$ is a two dimensional domain and $I$ is generated by 
homogeneous elements of the same degree, then  $\alpha(R, I) =  c^I({\bf m})$.
\end{propose}
\begin{proof}After Proposition~\ref{p2},
we only need to prove  $c^I({\bf m})\leq \alpha(R, I)$.
Let $I$ be generated by homogeneous generators $h_1, \ldots, h_{\mu}$ 
all of degree $d_0$.
By Proposition~2.12 of [T2], there is a constant $C_0$ such that for all $n\geq 0$ and for all
$x\in [1, \infty)$ (here $q=p^n$),
$$|f_n(x) - f_{n+1}(x)| \leq C_0/p^n.$$ Therefore there is a constant
$C_1$ such that $$|f_n(x) - f(x)| \leq C_1/p^n.$$

Let $x_0 = \alpha(R, I)$. Then, by Remark~\ref{r1}, we have  $x_0 > d_0$ 
(note here $q^{d-1} = q = p^n$) and
$$\frac{1}{p^n}\ell(R/{I}^{[q]})_{\lfloor xq\rfloor} =
f_n(x) \leq C_1/p^n,~~
\mbox{ for all}~~
x\geq x_0\quad\mbox{and}~~ n\geq 0.$$
Therefore
$$\ell(R/{I}^{[q]})_{\lfloor xq\rfloor}  \leq C_1,
~~\mbox{ for all}~~x\geq x_0\quad \mbox{and}~~~n\geq 0.$$
Let $X = {\rm Proj}~R$. Let
$$0\longrightarrow V \longrightarrow \oplus^{\mu}\sO_X\longrightarrow 
\sO_X(d_0)\longrightarrow 0,$$
be the short exact sequence of $\sO_X$-sheaves, where
$\oplus^{\mu} \sO_X\longrightarrow \sO_X(d_0)$ is the  multiplication 
map given by $(a_1, \ldots, a_{\mu}) \to \sum_{i}h_ia_i$ 
(here $h_i$ are the chosen generators).
This gives a long exact sequence of $\sO_X$-modules, for every $m\in \Z$,
$$0\longrightarrow H^0(X, (F^{n*}V)(m))
 \longrightarrow \oplus H^0(X, \sO_X(m))
 \longrightarrow  H^0(X, \sO_X(m+qd_0))
 \longrightarrow $$
$$\longrightarrow H^1(X, (F^{n*}V)(m))
 \longrightarrow \oplus H^1(X, \sO_X(m))
 \longrightarrow \cdots.$$

 We can choose
$m_1\geq 0$ such that, for all $m\geq m_1$,
$$H^1(X, \sO_X(m))=  0\quad{\rm
and}\quad H^0(X, \sO_X(m)) = R_m.$$
In particular, for $m\geq m_1$,
$$h^1(X, (F^{n*}V)(m)) = \ell(R/{I}^{[q]})_{m+qd_0}.$$
Since $\alpha(R, I) > d_0$ we can assume that $c^I({\bf m}) > d_0$ and
therefore we can also choose $q_0$
 large enough so that  we have  $(x_0-d_0)q_0 \geq m_1$.
Then for $q=p^n\geq q_0$,
\begin{equation}\label{ee2}
h^1(X, (F^{n*}V)(m)) = \ell(R/{I}^{[q]})_{m+qd_0} \leq C_1,
\quad\mbox{for all}\quad m\geq (x_0-d_0)q.
\end{equation}

 We choose  $z\in H^0(X, \sO_X(1))$ such that we have a short exact sequence
of $\sO_X$-modules
$$ 0\longrightarrow \sO_X(-1)\longrightarrow \sO_X\longrightarrow \sO_Y\longrightarrow 0,$$
which induces the long exact sequence of $k$ vector-spaces
$$0\longrightarrow H^0(X, F^{n*}V(m))
 \longrightarrow H^0(X, F^{n*}V(m+1))
\longrightarrow H^0(X, F^{n*}V(m+1)\mid_Y) $$
$$\longrightarrow H^1(X, F^{n*}V(m))
 \longrightarrow H^1(X, F^{n*}V(m+1))
\longrightarrow 0,$$
as $Y$ is a $0$-dimensional scheme,  for every $m \in \Z$, we have
$H^1(X, F^{n*}V(m)\mid_Y) = 0$ and
$H^0(X, F^{n*}V(m)\mid_Y) = C_2$, for some constant $C_2$.

\vspace{5pt}

\noindent{\bf Claim}.\quad $h^1(X, F^{n*}V(m)) = 0$,
for all $m \geq (x_0-d_0)q+C_1+1$, where $q\geq q_0$.

\vspace{5pt}

\noindent{\underline{Proof of the claim}}:\quad
Suppose  the map
$$ h_X(m): H^1(X, F^{n*}V(m))
 \longrightarrow H^1(X, F^{n*}V(m+1))$$ is an isomorphism for some 
  $m = m_0\geq 1$. Then
we prove that $h^1(X, F^{n*}V(m)) = 0$, for all $m\geq m_0$.
Note that  the map $h_X(m)$
is an isomorphism if and only if the canonical
map
$$H^0(X, F^{n*}V(m+1))
\longrightarrow H^0(X, F^{n*}V(m+1)\mid_Y) \longrightarrow 0$$
is surjective.
We have the following commutative diagram of canonical maps, where the
top horizontal map is surjective
$$\begin{array}{ccc}
H^0(X, F^{n*}V(m_0+1))\tensor H^0(X, \sO_X(1)) & \longrightarrow  &
H^0(X,F^{n*}V(m_0+1)\mid_Y)\tensor H^0(X,\sO_X(1))\\
\downarrow{} & & \downarrow\\
H^0(X, F^{n*}V(m_0+2)) & \longrightarrow &
H^0(X, F^{n*}V(m_0+2)\mid_Y)\end{array}$$
Moreover the second vertical map is surjective as $Y$ is $0$-dimensional and
$H^0(X, \sO_X(1))$ contains elements which do not vanish on $Y$.
Hence the bottom horizontal map is surjective.
By induction on $m$, it follows  that the map
$$H^1(X, F^{n*}V(m)) \longrightarrow H^1(X, F^{n*}V(m+1))$$
is an isomorphism, for all $m\geq m_0 $.
Since, by Serre vanishing  $H^1(X, F^{n*}V(m)) = 0$, for $m>>0$, 
we conclude  that
$h^1(X, F^{n*}V(m)) = 0$, for all $m\geq m_0$.

\vspace{5pt}

Now let $q\geq q_0$. Then, we claim that the map $h_X(m)$ is an isomorphism for some
$m\in [(x_0-d_0)q, (x_0-d_0)q+C_1)$: Otherwise
$$h^1(X, F^{n*}V(m))- h^1(X, F^{n*}V(m+1)) \geq 1,~~
\forall~~
 m\in [(x_0-d_0)q, (x_0-d_0)q+C_1+1),$$ which would imply that
$h^1(X, F^{n*}V({\lceil(x_0-d_0)q)\rceil}) > C_1$. But this
 contradicts the
inequality~(\ref{ee2}). Hence  the claim.

Therefore,
$$ h^1(X, F^{n*}V(m)) = \ell(R/I^{[q]})_{m+qd_0} = 0,
\quad \mbox{for all}\quad m\geq (x_0-d_0)q+C_1+1
\quad \mbox{and}\quad q\geq q_0.$$
$$\implies {\bf m}^{m+qd_0}\subset I^{[q]},\quad\mbox{for all}
\quad m+qd_0\geq x_0q+C_1+1\quad \mbox{and}\quad q\geq q_0$$

$$\implies c^I({\bf m}) \leq \lim_{q\to \infty}\frac{x_0q+C_1+1}{q} =
x_0 = \alpha(R, I).$$\end{proof}
 
\vspace{10pt}

In higher dimensional case we get such an equality when the ring is strongly 
$F$-regular on the punctured spectrum or when 
$\mbox{Proj}~R$ is smooth.
We recall the following definition from [HH2]. 

\begin{defn}\label{d6}A  Noetherian domain $R$ such that $R\longto R^{1/p}$ 
is module finite over $R$, is {\em strongly $F$-regular} if for every nonzero
$c\in R$ there exists $q$ such that $R$-linear map $R\longto R^{1/q}$ 
that sends $1$ to $c^{1/q}$ splits as a map of $R$-modules, {\em i.e.} 
iff $Rc^{1/q}\subseteq R^{1/q}$ splits over $R$.
A regular ring is strongly $F$-regular ([HH2]).
\end{defn}

\vspace{10pt}

Since the following lemma is easy to check we state it without the proof.

\begin{lemma}\label{freg} For a standard graded pair
$(R, I)$ and 
for a  fixed power $q_0$ of $p$
\[f_{R, I^{[q_0]}}(q_0x) = q_0^{d-1} f_{R,I}(x),\quad\mbox{for all}~~x\geq 0.\]
In particular  $\alpha(R, I^{[q_0]}) = q_0 \alpha(R,I)$.  
\end{lemma}

\begin{thm}\label{p5}Let $(R,I)$ be a standard graded pair of dimension 
$d\geq 2$.
$$R~~\mbox{is strongly}~F\mbox{-regular on}~~ \Spec R \setminus 
\{{\bf m}\} \implies
 \alpha(R,I) = c^{I}({\bf m}).$$
 
 Thus the above equality holds
 for a standard graded pair $(R, I)$, provided $R$ is 
a  domain and $X=\mbox{Proj}~R$ 
is strongly $F$-regular.
 
In particular the equality  $\alpha(R,I) = c^I({\bf m})$ holds 
 for any two dimensional standard graded pair $(R, I)$, where $R$ is a normal 
domain (equivalently $\mbox{Proj}~R$ is a nonsingular curve).

\end{thm}

\begin{proof}Is enough to prove that $\alpha(R,I) \geq c^{I}({\bf m})$.

By Theorem~5.10 of [HH2], there exists  $n_0$ 
such that ${\bf m}^{n_0}\subset \tau(R)$, where $\tau(R)$ denotes the test 
ideal of $R$.
Hence, for every ideal $J$ of $R$ we have
${\bf m}^{n_0}J^*\subseteq J$  where 
$J^*$ denotes the tight closure of $J$ in $R$.

It is enough to show that, if  $\beta \in \N[1/p]$ with  
 $\beta < c^{I}({\bf m})$ then
$\beta < \alpha(R,I)$. 
Let $2\epsilon = c^{I}({\bf m}) - \beta >0$.

Now we choose a power 
$q_0$ of $p$ such that, for $q\geq q_0$, we have
 $\beta q\in \N$, $\epsilon q \geq n_0$ and 
${\bf m}^{\beta q+\lfloor \epsilon q\rfloor} \not\subseteq I^{[q]}$.
In particular ${\bf m}^{\beta q_0+n_0}\not\subseteq I^{[q_0]}$ and therefore
${\bf m}^{\beta q_0}\not\subseteq [I^{[q_0]}]^*$.

We choose a homogeneous element 
$z \in {\bf m}^{\beta q_0} \setminus [I^{[q_0]}]^*$.  
Let  $J = (z, I^{[q_0]})$. By [HH1],  
$e_{HK}(R, I^{[q_0]}) 
- e_{HK}(R,J) >0$ and $e_{HK}(R, I^{[q_0]}) =  e_{HK}(R, [I^{[q_0]}]^*)$.
Therefore, by Remark~2.15 of [T2],
$$f_{R,[I^{[q_0]}]^*}(x) = f_{R,I^{[q_0]}}(x) \geq f_{R, J}(x),
\quad\mbox{for all}~~x.$$

Moreover, 
if  $x<\beta q_0 = \lfloor \beta q_0\rfloor$, 
then $$\deg~z^q\geq qq_0\beta \implies  (R/I^{[qq_0]})_{\lfloor xq\rfloor} = 
(R/(z^qR + I^{[qq_0]}))_{\lfloor xq\rfloor},\quad\mbox{for all}~~q.$$
Hence $f_{R,I^{[q_0]}}(x) = f_{R, J}(x)$, for $x<\beta q_0$.
In particular 
$$e_{HK}(R, I^{[q_0]}) 
- e_{HK}(R,J) = \int_{\beta q_0}^\infty (f_{R,I^{[q_0]}}(x) - 
f_{R, J}(x))dx >0.$$
This implies $\alpha(R,I^{[q_0]}) > \beta q_0$ and  by Lemma~\ref{freg},  
 $\alpha(R,I) > \beta$. Therefore $\alpha(R, I) \geq c^{I}({\bf m})$.

If $\mbox{Proj}~R$ is strongly $F$-regular then for any nonzero homogeneous 
element 
$c\in {\bf m}$, the ring $R_{(c)}$ (the subring consisting of degree $0$ elements 
of $R_c$)
 and therefore $R_c$ is strongly $F$-regular.
Hence  (by Theorem~5.10 of [HH2])
$\tau(R)$ is an ${\bf m}$-primary graded ideal.
Now the above arguement proves the second assertion.
\end{proof}

The equality $\alpha(R, I) = c^I({\bf m})$ also  holds when $I$ is generated by 
a system of parameters 
and then (not surprisingly) the number is 
decided by the degrees of the generators.

\begin{thm}\label{c11}Let $R$ be a standard graded pair of dimension 
 $d\geq 2$ over a perfect field of characteristic $p$ and 
 $I$ be 
generated by homogeneous system of parameters $f_1, \ldots, f_d$ of 
degrees $n_1,n_2, \ldots, n_d$ respectively. Then 
$$\alpha(R, I) = c^I({\bf m}) = n_1+n_2+\cdots + n_d.$$\end{thm}
\begin{proof} By Lemma~\ref{l1}, we have
$n_1+n_2+\cdots + n_d = \alpha(R, I)$.
Enough to prove the following claim.
\vspace{5pt}

\noindent{\bf Claim}.~~There exists a constant $l_0\in\N$ such that 
$\ell(R/(f_1^q, \ldots, f_d^q))_{\lfloor xq\rfloor}=0$, for 
all ${\lfloor xq\rfloor}\geq (n_1+\cdots+n_d)q+l_0$.
\vspace{5pt}

\noindent{\underline{Proof of the claim}}:~~We prove by induction on $d$. 
If $\dim~R=1$ then we  $k[f_1]\longto R$ is a finite graded map of degree $0$.
Hence $R$ is a direct sum of free and torsion modules over the 
principal ideal domain  $k[f_1]$. Therefore there exists
$l_0$ such that 
$\ell(R/(f_1^q))_{\lfloor xq\rfloor}=0$, for 
all ${\lfloor xq\rfloor}\geq n_1q+l_0$.
Rest of  the claim follows from  the surjective map 
(\ref{e8}) from Section~2.\end{proof}

As a corollary we have  bounds on $\alpha(R,I)$ in terms of the 
degrees of the generators of $I$. Moreover if 
$R$ is a polynomial ring then the   
 bound is explicit.

\begin{cor}\label{rsu}
If  $R$ is a standard graded $n$-dimensional 
ring and $I$ is generated by homogeneous elements 
$h_1, \ldots, h_s$ of degree $d_1\leq \cdots \leq d_s$ respectively.
Then
\begin{enumerate}
\item $d_1 \leq \alpha(R, I) \leq c^I({\bf m}) \leq d_1+\cdots+d_s$.
\item If 
 $R = k[X_1, \ldots, X_n]$ is a polynomial ring of dimension $\geq 2$
 then 
$$\alpha(R,I) = c^I({\bf m}) = 
\mbox{max}~\{s\mid {\bf m}^s\not\subseteq I\}+n.$$ 
\end{enumerate}
\end{cor}
\begin{proof}(1) We choose an ideal  $J$, which is
 generated by a system of parameters
$\{h_{i_1}, \ldots, h_{i_n}\}$ $\subset \{h_1,
\ldots, h_s\}$. Then, by    
Theorem~\ref{c11},
$$\alpha(R, I)   \leq \alpha(R, J) = c^J({\bf m}) = d_{i_1}+
\cdots + d_{i_n} \leq d_1+\cdots +d_s.$$

\noindent~(2)\quad 
The second assertion follows from Theorem~\ref{p5} and [HMTW], Example~2.7~(iii).
\end{proof}

The equality $\alpha(R,I) = c^I({\bf m})$ carries over to the Segre product of
standard graded pairs.

\begin{lemma}\label{l5}Let $(R_1, I_1), \ldots, (R_r, I_r)$ 
be standard graded pairs of dimension $\geq 2$, with their 
respective graded maximal ideals 
${\bf m}_1, \ldots, {\bf m}_r$ such that
 $c^{I_i}({\bf m}_i) = \alpha(R_i, I_i)$, for every $i$. Then 
$$ c^{I_1\#\cdots\#I_r}({{\bf m}_1\#\cdots\#{\bf m}_r}) = 
\alpha(R_1\#\cdots\#R_r, I_1\#\cdots \# I_r)$$ and therefore
$$c^{I_1\#\cdots\#I_r}({{\bf m}_1\#\cdots\#{\bf m}_r}) = 
\max\{c^{I_i}({\bf m}_i)\mid
1\leq i\leq r\}.$$
\end{lemma}
\begin{proof}Let $(R, I)$ and 
$(S, J)$  be  two standard graded pairs of dimensions $d_1$ and  $d_2$, 
respectively, over a  field $k$ of characteristic~$p >0$.
Then, by Proposition~2.17 of [T2]
$$f_{R\#S, I\#J} = (F_R)(f_{S, J}) + (F_S)(f_{R, I}) - 
(f_{R,I})(f_{S, J}),$$
where $F_R(x) = e(R)x^{d_1-1}/(d_1-1)!$ and 
$F_S(x) = e(S)x^{d_2-1}/(d_2-1)!$.

\vspace{5pt}

\noindent{\bf Claim}.\quad $\alpha(R\# S, I\# J) = 
\mbox{max}\{\alpha(R, I), \alpha(S, J)\}$.

\noindent{\underline{Proof of the claim}}:\quad Let $\alpha_1 = \alpha(R,I)$ and
$\alpha_2 = \alpha(S,J)$. Without loss of generality, one can assume that
 $\alpha_1\leq \alpha_2$.
It is enough to show $\alpha(R\# S, I\# J) \geq \alpha_2$. 
For any $\epsilon_1 >0$, there exists $0<\epsilon < \epsilon_1$ 
such that for 
$x\in [\alpha_2-\epsilon,~~\alpha_2)$, $f_{S,J}(x)\neq 0$ and
hence
$$f_{R\# S, I\# J}(x) = f_{R,I}(x)(F_S(x)-f_{S,J}(x)) + F_R(x)f_{S,J}(x) >0.$$
 
\vspace{5pt}

 Moreover, it follows from the definition that 
$c^{I\#J}({{\bf m}\#{\bf n}}) \leq \mbox{max}\{c^{I}({\bf m}), c^{J}({\bf n})\}$.
Hence, for $(R_i, {\bf m}_i)$ and $I_i$ as given above, we have
$$\begin{array}{lcl}
c^{I_1\#\cdots\#I_r}({{\bf m}_1\#\cdots\#{\bf m}_r}) & 
\leq & \max\{c^{I_i}({\bf m}_i)\mid
1\leq i\leq r\}\\\  
=  \max\{\alpha(R_i,I_i)\mid
1\leq i\leq r\} & = & \alpha(R_1\#\cdots\#R_r, I_1\#\cdots I_r).\end{array}$$
On the other hand, by Proposition~\ref{p2}, we have 
$$\alpha(R_1\#\cdots\#R_r, I_1\#\cdots \# I_r)\leq 
c^{I_1\#\cdots\#I_r}({{\bf m}_1\#\cdots\#{\bf m}_r}).$$
This proves the lemma. \end{proof}

We can summarize the results of this subsection in 
Theorem~E (stated in the introduction)

\vspace{5pt}
\noindent{\underline{\em Proof of Theorem~E}}.\quad It is a consequence of 
Proposition~\ref{p2}, Proposition~\ref{p3}, Theorem~\ref{p5},
 Theorem~\ref{c11} and
Lemma~\ref{l5}.$\Box$

\subsection{The shape of the graph of  $f_{R, {\bf m}}$  and regularity}

\begin{thm}\label{c1}Let $R$ be a standard graded ring of dimension 
$d\geq 2$ defined over a perfect field of characteristic $p>0$. 

Then 
$\alpha(R, {\bf m}) \leq d$.
Moreover, the equality holds
if and only if there exists a prime ideal $P$, such that 
$\dim~R/P = d$ and 
 $R/P$ is regular.

 In particular,  if $R$ is a 
domain then $\alpha(R, {\bf m}) = d$ if and only if $R$ is regular. 
\end{thm} 
\begin{proof}(1)\quad Let $J$ be a parameter ideal 
generated by elements of degree $1$. By Lemma~\ref{l1}, $\alpha(R, J) = d$. 
Hence  $\alpha(R, {\bf m}) \leq 
\alpha(R, J) \leq d$. 

\vspace{5pt}

\noindent~(2)\quad Suppose  $R/P$ is  regular for some $P\in \mbox{Assh}(R)$. 
Then, by Corollary~\ref{p1}, 
we have  $\alpha(R/P, {\bf m}/P) = d$.
Now the additivity of the HK density function implies that  
  $d\geq \alpha(R, {\bf m}) = \mbox{max}
\{\alpha(R/P, {\bf m}/P)\mid P\in \mbox{Assh}(R)\}\geq d$. 

Conversely, suppose $\alpha(R, {\bf m}) = d$. Let
 $P\in \mbox{Assh}(R)$ such that  $\alpha(R/P, {\bf m}/P) = d$.
Suppose $R/P$ is not regular. We choose a system of parameter 
ideal $J = (x_1, \ldots, x_d)$ of
linear forms such that $\{x_1, \ldots, x_d\}$ is a part of a minimal set of 
generators of ${\bf m}/P$. 
Then 
$$(J^\ast)\cap (R/P)_1 = J\cap (R/P)_1 \neq ({\bf m}/P)\cap (R/P)_1,$$
where the first equality follows from  Theorem~{2.2} of [S].

Hence, by Corollary~3.2 of [HMTW], we have $c^{{\bf m}/P}(J) < d$. 
On the other hand, by  Proposition~1.7 of [MTW], 
$c^{{\bf m}/P}(J) = c^{{\bf m}/P}({\bf m}/P)$, as
${\bf m}/P =$ the integral closure of $J$ in $R/P$.

But then 
 $\alpha(R/P, {\bf m}/P)\leq c^{{\bf m}/P}({\bf m}/P) < d$, 
which is a contradiction.

\vspace{5pt}

\noindent~(3)\quad Assertion~(3) is a particular case of  Assertion~(2).
\end{proof}

\begin{thm}\label{shk}Let $R$ be a standard graded 
domain of dimension $d\geq 2$. Then the function $f_{R, {\bf m}}$ is 
symmetric at $ x= d/2$, {\em i.e.}, 
$$f_{R, {\bf m}}(x) = f_{R, {\bf m}}(d-x),\quad\mbox{for all}~~x~~
\mbox{iff}\quad R~~~\mbox{is a regular ring}.$$

Moreover, if $\dim~R =2$ then 
either \begin{enumerate}
\item $f_{R, {\bf m}}$ is symmetric at $x=1$, or 
\item $f_{R, {\bf m}}(1-y) > f_{R, {\bf m}}(1+y)$, for all $y\in (0, 1)$. 
\end{enumerate}
\end{thm} 
\begin{proof}
If $R$ is regular then ${\bf m}$ is generated by a system of parameters of degree $1$ 
and therefore, by Corollary~\ref{p1}, the function  $f_{R, {\bf m}}$ is 
symmetric at $d/2$.

Conversely, if $f_{R, {\bf m}}$ is symmetric at $d/2$ then 
$\alpha(R, {\bf m}) = d$, hence, by Lemma~\ref{c1}, the ring $R$ is regular.

We give the proof of second assertion in the end of subsection~6.2

\end{proof}

In Section~6 we  study, in more detail, the invariant
$\alpha(R, I)$  for two dimensional domains. 
Here  we also explicitly express the HK density function $f_{R, I}$ in 
terms of the {\em strong Harder-Narasimhan slopes} of the associated  
syzygy bundles.  We recall some well known relevant results about vector bundles.

\section{Some Basic facts about semistability properties of vector bundles }
Let $X$ be a nonsingular projective curve over an algebraically  closed field $k$.
\begin{defn}\label{d5}A vector bundle $V$ on $X$ is 
{\em semistable} if for every proper subbundle $W\subset V$ we have 
$\mu(W)\leq \mu(V)$, where $\mu(V) = \deg(V)/\rank(V)$.

The bundle $V$ is {\em strongly semistable} (provided $\Char~k = p>0$) if, 
for every $m^{th}$ iterated Frobenius map 
$F^m:X\longto X$, the bundle $F^{m*}(V)$ is semistable, which is equivalent to 
the statement that $\pi^*V$ is semistable for every finite map $\pi:Y\longto X$
(because the pull back of a semistable bundle is semistable for any finite 
separable map).

Note that when  $\Char~k = 0$, then  $V$  semistable implies 
$\pi^*V$ is semistable for every finite map $\pi:Y\longto X$ (see [HL]).
Hence the notions of strongly semistability  and semistability coincide 
in $\Char~k = 0$.
\end{defn}

\begin{defn}\label{d3}Every vector bundle $V$ on  $X$ has the unique 
filtration of subbundles 
\begin{equation}\label{hn} 0 = F_0 \subset F_1 \subset \cdots \subset 
F_t \subset F_{t+1}= V,\end{equation}
called  {\it the Harder-Narasimhan (HN) filtration}  of $V$,
satisfying  the following conditions
\begin{enumerate}
\item for every $i$, the bundle $F_i/F_{i-1}$ is semistable and 
\item  
$\mu(F_1) > \mu(F_2/F_1)>\ldots >\mu(F_{t+1}/F_{t})$ 
($\equiv \mu(F_1) > \mu(F_2)>\ldots >\mu(V)$) .
\end{enumerate}

If in addition the bundles $F_i/F_{i-1}$ is 
strongly semistable, for every $i$ then 
 we call the filtration (\ref{hn}) {\em a strong
 Harder-Narasimhan filtration} (or a strong HN filtration). 
\end{defn}

\begin{notations}\label{hnd}Let $V$ be a vector bundle on a nonsingular 
projective curve $X$.
\begin{enumerate}
\item  If $V$ has the HN filtration as 
in (\ref{hn})
then we  define $\mu_i = \mu({F_i}/{F_{i-1}})$, the {\em HN slopes} of $V$ and 
$r_i = \rank({F_i}/{F_{i-1}})$, the HN
 ranks of $V$. We  
denote the minimum HN slope $\mu_{min}(V) = \mu({V}/{F_t})$.
We call the set 
$$(\{\mu_1, \mu_2, \cdots, \mu_{t+1}\}, \{r_1, 
\ldots, r_{t+1}\})$$ 
the {\em HN data} for  $V$.
\item By Theorem~2.7 of [L], if $\Char~k = p  >0$, then 
for a given bundle $V$, there exists $m>0$ 
such that 
the HN filtration of  $F^{m*}V$ is the strong HN filtration 
$$ 0 = E_0 \subset E_1 \subset \cdots \subset 
E_l \subset E_{l+1}= F^{m*}V.$$
In this case we define $a_i = (1/p^m)\mu({E_i}/{E_{i-1}})$ the {\em strong HN slopes} 
of $V$. We  denote the {\em minimum strong HN slope} and the  {\em strong HN data}
of $V$ by 
$$a_{min}(V) = (1/p^m)\mu({E_{l+1}}/{E_l})\quad\mbox{and}\quad 
\left(\{a_1, \ldots, a_{l+1}\},
 \{{\tilde r}_1, \ldots, {\tilde r}_{l+1}\}\right),$$
respectively, where ${\tilde r}_i = \rank(E_{i+1}/E_{i}).$
 
The notion of the strong HN slopes and the strong HN data
are well defined because
if $V$ has a strong HN filtration, then for every $n\geq 1$, 
the $n^{th}$ Frobenius 
pull back of the HN filtration of $V$ is the strong HN filtration for  $F^{n*}V$.
\end{enumerate}
\end{notations}

\begin{rmk}\label{r5}Let $\sO_X(1)$ be an ample line bundle of degree $d$ on 
$X$. Let ${\tilde E}$ be a semistable vector-bundle on $X$ with 
$\mu({\tilde E}) = \mu$ and $\rank({\tilde E}) = r$.
  Then by Serre duality 
$$\begin{array}{lcl}
 m < -{\mu}/{d} & \implies & h^1(X, {\tilde E}(m))
= -r(\mu+dm+(g-1))\\
-{\mu}/{d} \leq m \leq - {\mu}/{d} + (d-3) & \implies &
h^1(X, {\tilde E}(m)) =  C\\
-{\mu}/{d} + (d-3) < m & \implies &
h^1(X, {\tilde E}(m)) =  0,
\end{array}$$
where $|C| \leq r(g-1)$ and $g = \mbox{genus}(X)$.
\end{rmk}

\begin{rmk}\label{r4}Let $V$ be a vector bundle on $X$ with the
 HN filtration (\ref{hn}). 
\begin{enumerate}
\item If $V$ has a nontrivial HN filtration  ({\it i.e.} $V$ is not 
semistable) then $  \mu_{min}(V) < \mu(V)$.

 \item If $0\longto V'\longto V\longto V''\longto 0$ is a short exact 
sequence of nonzero 
vector bundles on $X$, then  
\begin{enumerate}
\item either $\mu(V') \leq \mu(V) \leq \mu(V'')$ 
or $\mu(V') \geq \mu(V) \geq \mu(V'')$. 
\item For a nonzero  map of bundles  $E\longto W$, 
 where $W$ is semistable, 
 $\mu_{min}(E) \leq \mu(W)$: We choose a semistable subquotient bundle $E_1$ 
(occuring in the HN filtration of $E$)  such that the restricted map $E_1
\longto W$ is nonzero. This gives $\mu_{min}(E) \leq \mu(E_1) \leq 
\mu(\mbox{Im}(E_1))\leq \mu(W)$.
\item $a_{min}(V) \leq \mu_{min}(V)$: Enough to prove $\mu_{min}(F^{m*}(V))
\leq p^m\mu_{min}(V)$, for any $m\geq 1$.
Let $V_1$ be the semistable quotient bundle of $V$ such that $\mu_{min}(V) = 
\mu(V_1)$. Let $W$ be the semistable quotient bundle of $F^{m*}(V_1)$ 
such that $\mu_{min}(F^{m*}V_1) = 
\mu(W)$.
Now the nonzero map $F^{m*}V\longto W$ gives
$$\mu_{min}(F^{m*}V) \leq \mu(W)
\leq \mu(F^{m*}(V_1)) = p^m\mu(V_1) = p^m\mu_{min}(V).$$
\end{enumerate}
\end{enumerate}
\end{rmk}

\section{The HK density functions for two dimensional rings}

\subsection{The HK density functions for vector bundles on curves}

Let $X$ be a nonsingular projective curve  over an 
algebraically closed  field of characteristic $p>0$. Let $\sO_X(1)$ be an
ample line bundle of degree $d$ on $X$.
Let $V$ be a  vector bundle on $X$. For the notion of HN data, (strong) 
HN slopes and the minimum strong HN slope $a_{min}(V)$, 
for a vector bundle $V$ on $X$, 
we refer to Notations~\ref{hnd}.

We define the HK density function of $V$ (a continuous but not necessarily 
compactly supported function) with respect to $\sO_X(1)$
(where $q=p^n$) 
\begin{equation}\label{fv}
f_{V, \sO_X(1)}:\R\longto [0, \infty)~~~\mbox{given by}~~~ 
x \to  \lim_{n\to \infty}\frac{1}{q}h^1(X, F^{n*}V(\lfloor (x-1)q\rfloor)).
\end{equation} 
This function is 
 well defined and is  given 
explicitly as follows. Let the strong HN data for $V$ be given by 
$(\{a_1, \ldots, a_{l+1}\}, \{r_1, \ldots, r_{l+1}\})$. 
Then there is $m_1$ such that 
$F^{m_1*}V$ has the strong HN filtration
$$0 = E_{0} \subset E_{1} \subset \cdots \subset E_{l} \subset E_{l+1} =
F^{m_1*}V, $$
where $a_i = (1/p^{m_1})\mu(E_i/E_{i-1})$ and 
$r_i  = \rank(E_i/E_{i-1})$. 

Since $a_1 > a_2 > \cdots > a_{l+1}$, we can choose $q>>0$ such that 
\begin{equation}\label{eem2}-\frac{ a_1q}{d} < 
-\frac{ a_1q}{d}+(d-3) < -\frac{a_2q}{d} < 
-\frac{ a_2q}{d}+(d-3) <\cdots < 
-\frac{ a_{l+1}q}{d}.\end{equation}

By Remark~\ref{r5}   
$$h^1(X, F^{n+m_1*}V(\lfloor (x-1)q\rfloor)) = 
\sum_{i=1}^{l+1}h^1(X, F^{n*}(E_i/E_{i-1})(\lfloor (x-1)q\rfloor)),$$
where $q=p^n$. Taking limit as $n\to \infty$, we get
$$\begin{array}{lcl}
 x < 1-a_1/d & \implies & f_{V,\sO_X(1)}(x) =
-\left[\sum_{i=1}^{l+1}a_ir_i+d(x-1)r_i\right] \\
1-a_i/d  \leq x < 1-a_{i+1}/d  & \implies & f_{V,\sO_X(1)}(x) =
- \left[\sum_{k={i+1}}^{l+1}a_kr_k+d(x-1)r_k\right]\\
1-a_l/d  \leq x < 1-a_{l+1}/d  & \implies & f_{V,\sO_X(1)}(x) =
- \left[a_{l+1}r_{l+1}+d(x-1)r_{l+1}\right]\\
1-a_{l+1}/d  \leq x   & \implies & f_{V,\sO_X(1)}(x) = 0.\end{array}$$
This implies  
$\mbox{Support}~f_{V,\sO_X(1)} \subseteq (-\infty, 1-a_{min}(V)/d]$
and $$\alpha(V, \sO_X(1)) := \mbox{Sup}~\{x\mid f_{V, \sO_X(1)}(x) >0\}
 = 1-\frac{a_{min}(V)}{d}.$$

\begin{rmk}Replacing $R$ by $R\tensor_k{\bar k}$ does not change the function
$f_{R,I}$ and the semistability behaviour of any vector bundle $V$ on 
$X= \mbox{Proj}~R$. Therefore   we can assume, without loss of generality, that the 
underlying field $k$ is algebraically closed.\end{rmk}

\subsection{The HK density functions of $(R,I)$ and its syzygy bundles} 

\begin{notations}\label{n3} Let $(R,I)$ be  a  standard graded pair, where $R$ is a domain defined over 
a perfect field. Let $h_1, \ldots, h_\mu$ be a set 
of generators of $I$ with $\deg~h_i = d_i$.

Let $S = \oplus_{m\geq 0} S_m$ be the integral closure of $R$ in its quotient field.
Since $R$ is a standard graded ring over $k$, the canonical embedding 
 $Y = {\rm Proj}~R \longto \P_k^n$ 
gives the very ample line bundle $\sO_Y(1)$ on  $Y$.
Let  $X = {\rm Proj}~S$ and let $\sO_X(1)$ be  the pull back of 
$\sO_Y(1)$ under the map $X\longto Y$. 
Let \begin{equation}\label{e2}
0\longto V \longto M = \oplus_{i=1}^{\mu}\sO_X(1-d_i)\longto \sO_X(1)
\longto 0,\end{equation}
be the canonical (locally split) exact
 sequence of locally free sheaves of $\sO_X$-modules,
where the map 
$\sO_X(1-d_i)\longto \sO_X(1)$ is given by the multiplication by the
the element $h_i$.

\end{notations}

Note that, in the case $I$ generated by degree $1$ elements,  the HK density 
function (and hence the maximum support 
of $f_{R, I}$) has been explicitly given in [T2].

\begin{thm}\label{n1}With the above notations, we have
\begin{enumerate}
\item $f_{R, I}(x) = 
f_{V, \sO_X(1)}(x)-f_{M, \sO_X(1)}(x), ~~\mbox{for}~~x\geq 0$ and
\item $\alpha(R, I) = 1- a_{i}(V)/d$ or $1- a_i(M)/d$, where $a_i(V)$
($a_i(M)$) is one of the  strong HN slopes of $V$ (respectively $M$) and 
hence $\alpha(R, I)$ is a rational number. Moreover
\item if $d_1 = \cdots = d_\mu$ then 
$ \alpha(R, I) = 1- a_{min}(V)/d$.
\end{enumerate}
\end{thm}
\begin{proof}We note that
 the  inclusion map $\pi:R\longrightarrow S$  is 
a graded finite map of degree $0$, where $S$ is a normal domain and $Q(R) = Q(S)$.
The additivity of the HK density function implies that 
$$f_{R, I}(x) = f_{S, I}(x) = \lim_{n\to \infty} f_n(S, I)(x) = 
\lim_{n\to \infty}\frac{1}{q} \ell\left(\frac{S}{I^{[q]}S}\right)_{\lfloor 
xq\rfloor}.$$
Also $X$ is a nonsingular projective curve.

Let $m_0 >0$ be such that, for $m\geq m_0$, we have  
$H^1(X, \sO_X(m)) = 0$ and $H^0(X, \sO_X(m)) = S_m$.
Then, for $m\geq m_0$, $n\geq 0$ and $q = p^n$, 
the  long exact sequence of cohomologies
\begin{equation}\label{*}0\longto H^0(X, (F^{n*}V)(m))\longto
 H^0(X, (F^{n*}M)(m))\longby{\varphi_{m,q}}
H^0(X, \sO_X(q+m))\end{equation}
$$ \longto H^1(X, (F^{n*}V)(m))\longto
 H^1(X, (F^{n*}M)(m))\longto 0$$
gives
 $$f_n\left(\frac{m+q}{q}\right) = \frac{1}{q} 
\ell\left(\frac{S}{I^{[q]}S}\right)_{m+q} = 
 \frac{1}{q}\left[h^1(X, (F^{n*}V)(m)) -
 h^1(X, (F^{n*}M)(m))\right].$$
Hence
$f_{R, I}(x) = 
f_{V, \sO_X(1)}(x)-f_{M, \sO_X(1)}(x)$, which proves the assertion (1).

The second assertion follows from the description of the HK density function
of a vector bundle in terms of its  strong HN data.

If $d_1= \ldots = d_{\mu}$ then $M$ is a strongly semistable bundle (being 
a sum of line bundles of the same degrees). 
Now, by Remark~\ref{r4} (as $\mu(M) < \mu(\sL)$)
$$a_{min}(V) \leq \mu_{min}(V) \leq  \mu(V) < \mu(M) = a_{min}(M).$$
 Therefore $\alpha(R, I) = \alpha(V, \sO_X(1)) = 1-a_{min}(V)/d$. 
\end{proof}

Now we are ready to give a proof of 
 
\vspace{10pt}
\noindent{\bf Theorem~\ref{shk}~(2)}\quad{\em 
If $R$ is a standard graded domain of dimension $2$ then 
either \begin{enumerate}
\item $f_{R, {\bf m}}$ is symmetric at $x=1$, or 
\item $f_{R, {\bf m}}(1-y) > f_{R, {\bf m}}(1+y)$, for all $y\in (0, 1)$. 
\end{enumerate}}

\begin{proof}
If $R$ is a regular ring then by Assertion~(1), the function $f_{R, {\bf m}}$ is 
symmetric at $x=1$. 

Let $R$ be a nonregular domain of dimension $2$. 

We consider 
the sequence~(\ref{e2}) 
$$0\longto V \longto M = \oplus^{s}\sO_X \longto \sO_X(1)\longto 0,$$
for  the pair $(R,{\bf m})$, where 
  $h_1, \ldots, h_s$ is a minimal set 
of degree $1$ generators of ${\bf m}$. 

Since $R$ is not regular,  $s\geq \mu({\bf m}) \geq 3$ and therefore 
$\rank~V =r =  s-1 \geq 2$.
If $(\{a_1>\ldots >a_{l+1}\}$,$ 
\{r_1, \ldots, r_{l+1}\})$ denotes  the strong HN data for 
$V$ then  $l+1\geq 1$. Also the strong HN data for $M$ is 
$(\{\mu(M)= 0\}, \{\rank(M) = s\})$. 
Hence 
$$y \in [0,  -a_1/d) \implies   
f_{R, {\bf m}}(1+y) = d-rdy < d-dy = f_{R, {\bf m}}(1-y).$$
Suppose there is $y_0\in (0, 1)$ such that 
$f_{R, {\bf m}}(1-y_0) \leq f_{R, {\bf m}}(1+y_0)$. 

Then 
$-a_j /d \leq y_0 < -a_{j+1}/d$, for 
some $1\leq j \leq l$.
Hence 
$$ -\sum_{i\geq j+1}(a_ir_i - 
dy_0 r_i) \geq d-dy_0 .$$
This gives
$$dy_0(1-\sum_{i\geq j+1}r_i) \geq d+\sum_{i\geq j+1}a_ir_i = 
-\sum_{i= 1}^{l+1}a_ir_i + \sum_{i\geq j+1}a_ir_i = -\sum_{i=1}^ja_ir_i.$$

Since $j+1 \leq l+1$ and $-a_i\geq 0$, for $i$, implies that 
both the sides of the above equation are $=0$.
In particular $j=l =1$ and $r_{l+1} =1 =r_2$  and
$a_1 =0$. 
Hence, if for some $m_1$,  $F^{m_1*}(V)$ has the strong HN filtration then it 
is given by
$$0=E_0 \subset E_1 \subset E_2 = F^{m_1*}V,$$
 where 
$\rank~E_1 = \rank~V -r_2 = \mu({\bf m})-2\geq 1$ and 
$\deg~E_1= a_1 = 0$.
Since  $E_1\subseteq \oplus^s \sO_X$, we deduce that $E_1$ is trivial vector 
bundle of rank $\geq 1$.
Hence $h^0({X}, F^{m_1*}V) \geq h^0({X}, E_1) \geq 1$.

On the other hand, 
the linear independence of the  $\{h_1,\ldots, h_s\} \in h^0(X, \sO_X(1))$ 
over the field  $k$ implies that  the set 
$\{h_1^q,\ldots, h_s^q\}$ $\in H^0(X, \sO_X(q))$
 is a linearly independent set over $k$.
Therefore the map (as in (\ref{*})) 
 $$\varphi_{0, q}: \oplus^s H^0(X, \sO_X)\longto H^0(X, 
\sO_X(q))$$
is injective. 
But then 
$h^0(X, F^{m_1*}V) =0$, which is a contradiction. Hence we conclude 
that  $f_{R, {\bf m}}(1+y) < f_{R, {\bf m}}(1-y)$, for every 
$y\in (0, 1)$.\end{proof}

\subsection{The $F$-threshold $c^I({\bf m})$ and $\alpha(R, I)$ in  
characteristic $0$}

\vspace{5pt}

We recall a well known notion of {\em spread}. We restrict our attention to 
the relevant situation.

\begin{defn}\label{ds} 
Let $X$ be a nonsingular projective curve over an algebraically 
 field of  $\Char~0$ and let $V$ be a
vector bundle on $X$. 
Then the triple $(A, X_A, V_A)$ is a {\em spread} for the pair $(X, V)$, 
if  $A \subset k $ is a finitely 
generated $\Z$-algebra and $X_A$ is a projective scheme over 
$A$, and $V_A$ is a coherent, locally free sheaf over $X_A$
with
$$ X_A\tensor_A k = X\quad \mbox{and}\quad V_A\tensor_A k = V.$$

  By [EGA~IV], we can further choose $A$  (if necessary, replacing $A$ by a 
finitely generated $\Z$-algebra $A_0$ such that $A\subseteq A_0 \subset k$)  
 so that $X_A$ is a smooth projective $A$-scheme. Moreover
if 
$$0\subset E_1 \subset \cdots \subset E_l \subset V$$ 
is the HN filtration of $V$ then $\{E_{iA}\}_i$ are coherent, locally 
free sheaves 
 with  a filtration 
$0\subset E_{1A}\subset \cdots \subset E_{lA} \subset V_A$ 
such that for 
all closed points $s\in \Spec~A$, 
$$0\subset E_{1(s)}\subset \cdots\subset E_{l(s)}\subset V_s$$
is the HN filtration of $V_s := 
V_A\tensor_A{\overline {k(s)}}$ as sheaves over $X_s := 
X_A\tensor_A{\overline {k(s)}}$, where 
$E_{is} := 
E_{iA}\tensor_A{\overline {k(s)}}$ (this follows by 
the openness property of semistable vector bundles proved in [Ma]).

Similarly for a  pair $(R, I)$, where $R$ is a finitely generated $\N$ 
(standard) graded 
two dimensional domain  over a field of 
characteristic $0$ and $I\subset R$ is a homogeneous ideal of finite 
colength, then the triple $(A, R_A, I_A)$ is a {\em spread} of $(R, I)$, 
if $A$  is a 
finitely generated $\Z$-algebra
$A \subset k$ and $R_A$ is a finitely generated $\N$ (standard) graded algebra over 
$A$ and  $I_A\subset R_A$ is a homogeneous ideal such that 
$$R_A\tensor_A k = R\quad \mbox{and}\quad  
\mbox{Im}~(I_A\tensor_Ak) (\subset R) = I.$$ 
We can  choose such  an $A$ so that for any 
closed point $s\in \Spec~A$ the ring  $R_s = R_A\tensor_A {\overline{k(s)}}$ 
is a finitely generated $\N$ (standard) graded $2$-dimensional domain (which is a normal 
domain if $R$ is normal) over ${\overline{k(s)}}$ and the ideal $I_s  = 
\mbox{Im}~(I_A\tensor_A{\overline{k(s)}}) \subset
R_s$  is a homogeneous ideal of finite colength. 

Moreover, if $(A, X_A, V_A)$ is 
a spread for $(X, V)$
and $(A, R_A, I_A)$ is  a spread for $(R, I)$ then 
for any finitely generated $\Z$-algebra $A'$ with 
$A\subset A'\subset k$, the triple  
$(A', X_{A'}, V_{A'})$ ($(A', R_{A'}, I_{A'})$) also satisfies the same 
properties as 
 $(A, X_A, V_A)$ ($(A, R_A, I_A)$ respectively). Hence we may always assume 
that the spread (as above) is chosen  such that $A$ contains a given 
finitely generated $\Z$-algebra $A_0\subseteq k$. \end{defn}

\begin{notations}\label{ex5}Let $R$ be a standard graded domain over a field $k$ of 
characteristic $0$
with a homogeneous ideal $I$ of finite colength such that $I$ is generated by 
the same degree elements. Let $S$ be the integral closure.
Also let $X = \mbox{Proj}~S$ be the  nonsingular projective curve over $k$ 
with a vector bundle $V$ given by the canonical  exact sequence 
of the sheaves of $\sO_X$-modules
\begin{equation}\label{e10}
0\longto V\longto \oplus^\mu\sO_{X}(1-d_0)\longto \sO_{X}(1)\longto 
0,\end{equation}
where $I$ is generated by $\mu$ homogeneous generators of  
degree $d_0$ each.

For $(X, V)$, $(R, I)$ and $(S, IS)$, we can,  respectivey, choose spreads
$(A, X_A, V_A)$, $(A, R_A, I_A)$, $(A, S_A, IS_A)$,  as above, 
 such that, in addition, 
for every $s\in \mbox{Spec}~A$,
$S_s$ is the intergral closure of $R_s$ and
$V_s$ is the syzygy bundle given by 
 the  canonical short exact sequence
 of locally free sheaves of $\sO_{X_s}$-modules
\begin{equation}\label{e1}
0\longto V_s\longto \oplus^\mu\sO_{X_s}(1-d_0)\longto \sO_{X_s}(1)\longto 
0.\end{equation}
\end{notations}

\begin{rmk}\label{r6}If  ${\tilde E} = V$ or ${\tilde E} =
 \oplus^\mu\sO_{X_s}(1-d_0)$ is a  vector bundle on 
$X$ and  
$(A, X_A, {\tilde E}_A)$ is a spread for $(X, {\tilde E})$ (as given  above). Then 
(see the proof of Theorem~4.6 in [T4])  
$$f^{\infty}_{{\tilde E},\sO_X(1)} := \lim_{p_s\to \infty}
f_{{{\tilde E}_s}, \sO_{X_s}(1)}~~\mbox{exists},$$
 where
$s\in \mbox{Spec}(A)$ are closed points and $p_s =$ characteristic of $R_s$.
Moreover, if 
$$ 0 = E_{0} \subset E_{1} \subset \cdots \subset E_{l} \subset E_{l+1} =
{\tilde E} $$ 
is the HN filtration of ${\tilde E}$ such that 
$\mu_i = \mu(E_i/E_{i-1})$ and $r_i = \rank(E_i/E_{i-1})$ then 
$$\begin{array}{lcl}
1 \leq x < 1-\mu_1/d & \implies & f^{\infty}_{{\tilde E}, \sO_{X}(1)}(x) =
-\left[\sum_{i\geq 1}\mu_ir_i+d(x-1)r_i\right] \\
1-\mu_i/d  \leq x < 1-\mu_{i+1}/d  & \implies & 
f^{\infty}_{{\tilde E}, \sO_{X}(1)}(x) =
- \left[\sum_{k\geq {i+1}}\mu_kr_k+d(x-1)r_k\right].\end{array}$$
\end{rmk}

\begin{thm}\label{vb1}Let $(R, I)$ be a standard graded pair 
and $(A, R_A, I_A)$ and $(A, X_A, V_A)$ be associated spreads
as   
in Notations~\ref{ex5}.
Let  $s$ denote a closed point of $\mbox{Spec}(A)$ and $p_s = \Char~R_s$. Then
\begin{enumerate}
\item  for every $x\geq 0$, 
$f^{\infty}_{R, I}(x) := \lim_{p_s\to \infty}f_{R_s, I_s}(x)$ exists and
the function\linebreak $f^{\infty}_{R, I}:[0, \infty)\longto [0, 
\infty)$ is a continuous compactly supported  function 
with 
 $$\alpha^{\infty}(R,I) := \mbox{Sup}~\{x\mid f^\infty_{R, I}(x) \neq 0\} =
1- \frac{\mu_{min}(V)}{d}.$$
\item $\lim_{p_s\to \infty}\alpha(R_s, I_s) = \alpha^{\infty}(R,I)$. 
\item  $\alpha(R_s, I_s) \geq \alpha^{\infty}(R, I)$, for $p_s>>0$.
If, in addition, the bundle $V$ is semistable then
$$\alpha(R_s, I_s) = \alpha^{\infty}(R, I)~~
\iff~~ V_{s}~~\mbox{is
strongly semistable}.$$

\end{enumerate}
\end{thm}
\begin{proof}Let $M = \oplus^\mu\sO_X(1-d_0)$.
\vspace{5pt}

\noindent{(1)}\quad By Theorem~\ref{n1}, $f_{R_s, I_s}(x) = f_{{V_s}, \sO_{X_s}(1)}(x)
-f_{{M_s}, \sO_{X_s}(1)}(x)$. 
Therefore (Remark~\ref{r6})
$$f^\infty_{R,I}(x) := \lim_{p_s\to \infty}f_{R_s, I_s}(x)
= f^\infty_{V,\sO_{X}(1)}(x) - f^\infty_{M,\sO_{X}(1)}(x),$$
and hence is a continuous compactly supported  function. Moreover 
 we can write the functions $f^\infty_{V,\sO_{X}(1)}$ and 
$f^\infty_{M,\sO_{X}(1)}$ in terms of the HN data of $V$ and $M$ respectively.
Since 
$\mu_{min}(V) \leq \mu(V) < \mu(M) = \mu_{min}(M)$ (by Remark~\ref{r4}), we 
have $\alpha^\infty(R,I) = 1- {\mu_{min}(V)}/{d}$.

\vspace{5pt}
\noindent{(2)}\quad The second assertion follows by Lemma~1.16 of [T1] 
which asserts
 that  
$\mu_{min}(V) = \lim_{p_s\to \infty}a_{min}(V_s)$.

\vspace{5pt}

\noindent{(3)}\quad For $p_s>>0$ we have  $a_{min}(V_s) \leq 
\mu_{min}(V_s) = \mu_{min}(V)$ (Remark~\ref{r4}). By
Theorem~\ref{n1},  $\alpha(R_s, I_s) =   1-{a_{min}(V_s)}/{d}$, which implies 
$\alpha(R_s, I_s)\geq \alpha^{\infty}(R, I)$.

If the bundle $V$ is semistable then, for $p_s>>0$,   
$\mu_{min}(V) = \mu(V) = \mu(V_s)$. Now $\alpha(R_s, I_s) = 
\alpha^{\infty}(R, I)$ if and only if 
$a_{min}(V_s) = \mu(V)$. 
On the other hand $a_{min}(V_s) = \mu(V_s)$ if and only if the bundle
  $V_s$ is strongly 
semistable on $X_s$.
\end{proof}

\begin{rmk}(1)\quad Theorem~\ref{vb1}~(2) does not 
straightaway follow from Theorem~\ref{vb1}~(1). In fact
there   exist uniformly convergent 
sequences  $\{g_s:\R_{\geq 0}\longto \R_{\geq 0}\}_{s\in \N}$ 
 consisting of  compactly supported 
continuous functions, converging  to a function 
$g:\R_{\geq 0}\longto \R_{\geq 0}$ but  
$$\lim_{s\to \infty} \mbox{Sup}\{x\mid g_s(x)\neq 0\} > 
\mbox{Sup}\{x\mid g(x)\neq 0\}.$$
  
{\em e.g.},  let 
$\{g_s:[0, \infty)\longto [0, \infty)\}$ be the set of functios given by
$$g_s(x) = \frac{x}{s^2},~~\mbox{if}~~ x\in [0,1], \quad 
g_s(x) = \frac{2-x}{s^2},~~\mbox{if}~~x\in [1, 2] \quad{and}\quad 
g_s(x) = 0,~~\mbox{if}~~x\geq 2.$$

\noindent\quad{(2)}\quad
Note that if
$I$ has homogeneous generators of degree $1$, then Theorem~\ref{vb1}~(2) is a consequence of
 Theorem~4.6 of [T4].
\end{rmk}

\vspace{10pt}

\begin{cor}\label{pc1}Let  $(R, I)$ be a pair as 
in Theorem~\ref{vb1}. If $\deg\sO_X(1) > 2\mbox{genus}(X)$, where
 $X=\mbox{Proj~R}$, then 
 for $p_s>>0$, 
$$\alpha(R_s, I_s) = \alpha^{\infty}(R, I)~ \iff V_{s}~~\mbox{is 
strongly semistable}.$$
\end{cor}
\begin{proof}Since, by  [PR] and Lemma~2.1 of [T2], the 
  bundle $V$ is semistable, the proof follows from Theorem~\ref{vb1}~(3).\end{proof}

\noindent{\underline{\em{Proof of Theorem~C and Corollary~D}}.\quad Theorem~C 
and Corollary~D  follow from 
Theorem~\ref{vb1} and Corollary~\ref{pc1} (respectively) as $c^I({\bf m}) 
= \alpha(R, I)$, by Proposition~\ref{p3}.\hspace{5pt}$\Box$

\vspace{10pt}

\begin{rmk}
 We will see in the next section (Theorem~\ref{?}) that the set of  
primes $p_s$, where
the equality $c^{{\bf m}_s}({\bf m}_s) = c^{\bf m}_{\infty}({\bf m})$ holds, 
is always a  
dense (but not necessarily open) set whenever $X$ is an irreducible plane 
trinomial curve. \end{rmk}

\section{$F$-thresholds  of plane trinomials}
For a pair $(R, I)$, where $R$ is a two dimensional domain and $I$ 
is generated by homogeneous elements of the same degree,  the F-threshold 
$c^I({\bf m})$ is given  in terms of the minimum  
strong HN slope of the  syzygy bundle $V$ (by Theorem~B). On the other 
hand, if $R$ is a  irreducible plane trinomial curve 
of degree $d\geq 3$ and $I(n) = (x^n, y^n, z^n)$ then all the 
strong HN slopes of the syzygy bundle $V$  can be computed  due to  
  a group theoretic interpretations of the complete 
strong HN data of the syzygy bundle (Theorem~3.5 in [T3]). Recall  
$$0\longto V\longto \oplus\sO_X(1-n)\longto \sO_X(1)\longto 0$$
is the canonical short exact sequence of $\sO_X$-modules and  
$$\alpha(R, I(n)) = c^{I(n)}({\bf m}) = 1-a_{min}(V)/d.$$

Note that such a plane curve is either `regular' or `irregular' 
(this terminology is taken from [M2]), where 
 a plane curve is called irregular if it has a singular point of 
multiplicity $r\geq d/2$, otherwise it is called regular.

\begin{thm} If $R = k[x,y,z]/(h)$, where $h$ is an irregular 
trinomial of degree $d$ and 
multiplicity $r$ (therefore $r\geq d/2$). Then for ${\bf m} = (x,y,z)$ and 
$I(n) = (x^n, y^n, z^n)$, where $n\geq 1$, we have 
$$\alpha(R, I(n)) = c^{I(n)}({\bf m}) = \frac{n+2}{2}+
\left(\frac{(2r-d)n}{2d}\right)^2.$$
\end{thm}
\begin{proof}Follows from Theorem~1.1 in [T3].\end{proof}
For regular trinomial plane curves, we recall the following notations
from [M2].

\begin{notations}\label{nt} Given a regular trinomial $h$ of degree $d$ 
in $k[x,y,z]$ (upto 
linear change of variables, any such 
trinomial is of type~I or type~II, as given below),
we can associate positive integers
$\alpha, \beta, \nu, \lambda >0$ as follows:

\begin{enumerate}
\item Type~(I)~$h = x^{a_1}y^{a_2} + y^{b_1}z^{b_2} +z^{c_1}x^{c_2}$, we denote
$$\alpha = a_1+b_1-d,~~ \beta = a_1+c_1-d,~~ \nu = b_1+c_1-d,~~
\lambda = a_1b_1+a_2c_2-b_1c_2.$$
\item Type~(II)~ $h = x^d + x^{a_1}y^{a_2}z^{a_3} + y^bz^c$,
 we denote
$$\alpha = a_2, \beta = c,
\nu = a_2+c-d\quad\mbox{and}\quad \lambda = a_2c-a_3b.$$
\end{enumerate}
  Moreover we denote
\begin{equation}\label{th}t_h = 
(\alpha/\lambda, \beta/\lambda, \nu/\lambda),\quad\mbox{and}\quad
a=\mbox{gcd}(\alpha, \beta, \nu, \lambda)~~\mbox{and}~~\lambda_h = 
\lambda/a.\end{equation}
\end{notations}

\begin{defn}\label{dd1} For a given regular trinomial $h$, 
we recall the following definition given in
[HM] and [M2].
Let $L_{odd} = \{u=(u_1, u_2, u_3)\in \Z^3\mid
\sum_i u_i~~~~\mbox{odd}\}$.
For any $u\in L_{odd}$ and for $l,s\in \Z$ and $n\geq 1$,
the {\em taxicab distance} 
$$\mbox{Td}(l^st_hn,u) = \mbox{Td}((\frac{l^s\alpha n}{\lambda}, 
\frac{l^s\beta n}{\lambda}, \frac{l^s\nu n}{\lambda}), (u_1, u_2, u_3)) = 
|\frac{l^s\alpha n}{\lambda}-u_1|+
|\frac{l^s\beta n}{\lambda}-u_2| + |\frac{l^s\nu n}{\lambda}-u_3|.$$

For a given regular trinomial $h$ and a given $n\geq 1$, let 
(1)~~$D_n(l) :=  s\geq 0$ if   
 $s$ is the smallest integer, for which the inequality 
 $\mbox{Td}(l^st_hn, u)< 1$ has a solution for some $u\in L_{odd}$ and 
in that case 
define $T_n(l) = \mbox{Td}(l^st_hn, u)$.
(2) If there is no such $s$  then  we define $T_n(l)= 1$ and $D_n(l) = \infty$.

\end{defn}

We recall the following result (Theorem~3.5 in [T3]).

\begin{thm}\label{r*} For a given  regular trinomial  $h\in k[x,y,z]$ over a 
field of $\Char~p >0$ and given
 $n\geq 1$, 
there is a well defined set theoretic map:
$$\Delta_{h,n}:\frac{(\Z/2\lambda_h\Z)^*}{\{1,-1\}}\longto
\left\{\frac{1}{\lambda_h}, \frac{2}{\lambda_h}, 
\ldots, \frac{\lambda_h-1}{\lambda_h}
\right\}
\times \{0, 1, \ldots, \phi(2\lambda_h)-1\} \bigcup \{(1, \infty)\},$$
given by $l\to (T_n(l), D_n(l))$,
where $D_n(l)$ and $T_n(l)$ are as in  Definition~\ref{dd1}.
Moreover 
\begin{enumerate}
\item $\Delta_{h,n} \equiv \Delta_{h,n+2\lambda_h}$.
\item Either $D_n(l) = \infty$ or  $D_n(l) < $
the order of the element $l$ in the group
$(\Z/2\lambda_h\Z)^*$.
\end{enumerate}
\end{thm}

Following result and explicit examples can be obtained easily from
Lemma~5.4 (in [T3]). 

\begin{thm}\label{?}Let $R = k[x,y,z]/(h)$ and  ${\bf m} = (x,y,z)$
 where  $h$ is a regular 
trinomial of degree $d$ over a 
field of $\Char~p >0$. 
Let 
 $\Delta_{h,n}$ be the  set theoretic map
 given as in Theorem~\ref{r*}. Then, for 
 $p\geq \max\{n, d^2\}$ and  $I(n) = (x^n, y^n, z^n)$, where $n\geq 1$,  
we have the following:
\begin{enumerate}
\item $p\equiv\pm 1\pmod{2\lambda_h}$ then 
$$\alpha(R, I(n)) = c^{I(n)}({\bf m}) = \frac{n+2}{2}.$$
\item If $p\equiv\pm l\pmod{2\lambda_h}$  and
$\Delta_{h,n}(l) = (T_n(l), D_n(l))$ then
$$\alpha(R, I(n)) = c^{I(n)}({\bf m}) = \frac{n+2}{2}+
\frac{\lambda(1-T_n(l))}{2p^{D_n(l)}d},$$
and the integers $\lambda(1-T_n(l))$ and $D_n(l)$ are bounded in terms of the 
exponents of the trinomal $h$ such that 
\begin{enumerate}
\item $\lambda(1-T_n(l)) \in \{0, \cdots, \lambda-1\}$  and 
\item $D_n(l) \in \{0, \ldots, O(l)\}$, where
 $O(l)$ is the order of 
the element $l$ in the group $(\Z/2\lambda_h\Z)^*\}$. 
\end{enumerate}
\end{enumerate}
\end{thm}

\begin{rmk} Theorem~\ref{r*} 
implies that for 
$p\geq d^2$ with $p\equiv\pm 1\pmod{2\lambda_h}$ 
(hence infinitly many primes $p$), we have 
$$\alpha(R, {\bf m}) = \alpha^{\infty}(R, {\bf m})
 = c^{\bf m}({\bf m}) = c^{\bf m}_{\infty}({\bf m}) = {3}/{2} .$$

Moreover, for any given explicit trinomial $h$ and an integer 
$n\geq 1$, Theorem~\ref{r*} gives an 
effective alogorithm to compute $\alpha(R, I(n))$ 
and $c^{I(n)}({\bf m})$, as we need to check if 
 the taxicab distance
 $\mbox{Td}(l^stn, u)< 1$ has a solution for some $u\in L_{odd}$, 
where  $0\leq s <\phi(2\lambda)$ (that means for finitely many $s$).

\end{rmk}

\begin{cor}\label{c4}Let $S_1, \ldots, S_r$ be a set of irreducible plane 
trinomial curves 
of degrees $\geq 4$ defined over a field of characteristic $0$. Then 
for infinitely many primes $p=\mbox{characteristic}~k(s)$, where $s$ denotes a 
closed point of $\mbox{Spec}~(A)$ we have  
$$c^{({\bf m}_{1}\#\cdots\#{\bf m}_{r})_{s}}(({\bf m}_{1}\#\cdots\#{\bf m}_{r})_{s}) 
= c_\infty^{({\bf m}_{1}\#\cdots\#{\bf m}_{r})}
({\bf m}_{1}\#\cdots\#{\bf m}_{r}).$$ 
If one of the trinomial is symmetric ({\it i.e.}, $h = 
x^ay^{d-a}+ y^az^{d-a} + z^ax^{d-a}$, where $0\leq a\leq d$)
of degree $d\neq 5$ then there 
are also infinitely many primes, for which 
$$c^{({\bf m}_{1}\#\cdots\#{\bf m}_{r})_{s}}(({\bf m}_{1}\#\cdots\#{\bf m}_{r})_{s}) 
> c_\infty^{({\bf m}_{1}\#\cdots\#{\bf m}_{r})}
({\bf m}_{1}\#\cdots\#{\bf m}_{r}).$$ 
\end{cor}

\begin{ex}\label{ex1}$$\mbox{Let}\quad 
R = \frac{k_p[x,y,z]}{(x^d+y^d+z^d)}, \quad{\bf m} = (x,y,z)$$
where $k_p$ denotes a field  of characteristic $p\geq d^2$.
Then by Corollary~4.4 and Theorem~4.5 of [T3] 
\begin{enumerate}
\item if $d > 5$ is an odd integer then  $p \equiv \pm (d+2)\pmod{2d}$
implies $$\alpha(R, {\bf m}) = c^{\bf m}({\bf m}) =  \frac{3}{2} +\frac{d-6}{2dp}.$$
 \item If $d\geq 4$ is an even integer  
 then 
\begin{enumerate} \item $p\equiv \pm(d+1)\pmod{2d}$
implies $$\alpha(R, {\bf m}) = c^{\bf m}({\bf m}) =  \frac{3}{2} +\frac{d-3}{2dp}.$$
 \item and $p\equiv \pm 1\pmod{{2d}}$
implies $$\alpha(R, {\bf m}) = c^{\bf m}({\bf m}) =  \frac{3}{2}.$$
 \end{enumerate}
\end{enumerate}
\end{ex}

\begin{ex}\label{ex2}Let 
$$R = \frac{k_p[x,y,z]}{(x^{d-1}y+y^{d-1}z+z^{d-1}x)},\quad\mbox{and}\quad{\bf m} = (x,y,z)$$
where $k_p$ denotes a field  of characteristic $p\geq d^2$.
By Corollary~4.4 and Theorem~4.5 of [T3], where 
$\lambda = (d^2-3d+3)$
\begin{enumerate}
\item if $p\equiv \pm 1\pmod{\lambda}$
then $$\alpha(R, {\bf m}) = c^{\bf m}({\bf m}) =  \frac{3}{2}.$$
\item If $p \equiv \lambda \pm 2\pmod{2\lambda}$
and
\begin{enumerate}\item
if $d \geq 6$ is an even integer then 
\begin{enumerate}
\item 
for $3\cdot 2^{m-2}\leq d-1 < 2^m$, where  $m\geq 1$, we have 
 $$\alpha(R, {\bf m}) = c^{\bf m}({\bf m}) =  \frac{3}{2} +
\frac{2(d-2)(d-1-3\cdot 2^{m-2})+2}{dp^m},$$
 \item for $2^{m}\leq d-1 < 
3\cdot 2^{m-1}$, where $m\geq 1$, we have
 $$\alpha(R, {\bf m}) = c^{\bf m}({\bf m}) =  \frac{3}{2} +
\frac{(d-2)(3\cdot 2^{m-1}-(d-1))-1}{dp^m}.$$
\end{enumerate}
\end{enumerate}
\begin{enumerate}\item If $d\geq 7$ is an odd integer then  
$$\alpha(R, {\bf m}) = c^{\bf m}({\bf m}) =  \frac{3}{2} +
\frac{\lambda -6(d-2)}{2dp}.$$
\item If $d=5$ then 
$$\alpha(R, {\bf m}) = c^{\bf m}({\bf m}) =  \frac{3}{2} +\frac{7}{2dp^3}.$$
\end{enumerate}
\end{enumerate}
\end{ex}

\end{document}